\documentclass{amsart}

\usepackage[T1]{fontenc}
\usepackage[english]{babel}
\usepackage{amsmath}
\usepackage{amsthm}
\usepackage{amssymb}
\usepackage{amsfonts}
\usepackage{amsxtra}
\usepackage{color}
\usepackage[breaklinks]{hyperref}
\usepackage{enumerate}
\usepackage{multicol}

\numberwithin{equation}{section}

\newcommand{\td}{\,\mathrm{d}}
\newcommand{\Ad}{\textup{Ad}}
\newcommand{\ad}{\textup{ad}}

\newcommand{\id}{\textup{id}}

\newcommand{\Skew}{\textup{Skew}}

\newcommand{\GL}{\textup{GL}}
\newcommand{\gl}{\mathfrak{gl}}
\newcommand{\SL}{\textup{SL}}
\renewcommand{\sl}{\mathfrak{sl}}

\renewcommand{\sp}{\mathfrak{sp}}

\newcommand{\SO}{\textup{SO}}
\newcommand{\so}{\mathfrak{so}}

\newcommand{\su}{\mathfrak{su}}

\newcommand{\RR}{\mathbb{R}}
\newcommand{\KK}{\mathbb{K}}
\newcommand{\CC}{\mathbb{C}}
\newcommand{\ZZ}{\mathbb{Z}}
\newcommand{\NN}{\mathbb{N}}

\newcommand{\HH}{\mathbb{H}}
\newcommand{\OO}{\mathbb{O}}

\newcommand{\1}{\mathbf{1}}
\newcommand{\Ind}{\textup{Ind}}

\newcommand{\calO}{\mathcal{O}}

\newcommand{\calH}{\mathcal{H}}
\newcommand{\calP}{\mathcal{P}}
\newcommand{\calU}{\mathcal{U}}

\newcommand{\calV}{\mathcal{V}}

\newcommand{\calR}{\mathcal{R}}

\newcommand{\frakg}{\mathfrak{g}}
\newcommand{\fraku}{\mathfrak{u}}
\newcommand{\frakk}{\mathfrak{k}}
\newcommand{\frakp}{\mathfrak{p}}
\newcommand{\fraks}{\mathfrak{s}}
\newcommand{\frakn}{\mathfrak{n}}
\newcommand{\fraka}{\mathfrak{a}}

\newcommand{\frakm}{\mathfrak{m}}
\newcommand{\frakl}{\mathfrak{l}}
\newcommand{\frakt}{\mathfrak{t}}

\newcommand{\frakh}{\mathfrak{h}}
\newcommand{\frake}{\mathfrak{e}}
\newcommand{\frakf}{\mathfrak{f}}

\newcommand{\frakz}{\mathfrak{z}}

\newcommand{\Ann}{\textup{Ann}}

\newcommand{\GKdim}{\textup{GK-dim}}
\DeclareMathOperator{\gr}{gr}

\newcommand{\Set}[2]{\left\{#1\,\middle|\,#2\right\}} 
\newcommand{\set}[2]{\{#1\,|\,#2\}}

\newcommand{\blank}{\,\cdot\,}

\theoremstyle{plain}
\newtheorem{theorem}{Theorem}[section]
\newtheorem{proposition}[theorem]{Proposition}
\newtheorem{lemma}[theorem]{Lemma}
\newtheorem{corollary}[theorem]{Corollary}

\theoremstyle{definition}

\newtheorem{remark}[theorem]{Remark}

\begin{document}

\title[Degenerate principal series on symmetric $R$-spaces]{Structure of the degenerate principal series on symmetric $R$-spaces and small representations}
\date{April 5, 2013}

\author{Jan M\"{o}llers}

\address{Institut for Matematiske Fag, Aarhus Universitet, Ny Munkegade 118, 8000 Aarhus C, Denmark}
\curraddr{}
\email{moellers@imf.au.dk}

\author{Benjamin Schwarz}
\address{Institut f\"{u}r Mathematik, Universit\"{a}t Paderborn, Warburger Str. 100, 33098 Paderborn, Germany}
\curraddr{}
\email{bschwarz@math.upb.de}
\thanks{}

\begin{abstract}
Let $G$ be a simple real Lie group with maximal parabolic subgroup $P$ whose nilradical is abelian. Then $X=G/P$ is called a symmetric $R$-space. We study the degenerate principal series representations of $G$ on $C^\infty(X)$ in the case where $P$ is not conjugate to its opposite parabolic. We find the points of reducibility, the composition series and all unitarizable constituents. Among the unitarizable constituents we identify some small representations having as associated variety the minimal nilpotent $K_\CC$-orbit in $\frakp_\CC^*$, where $K_\CC$ is the complexification of a maximal compact subgroup $K\subseteq G$ and $\frakg=\frakk+\frakp$ the corresponding Cartan decomposition.
\end{abstract}

\subjclass[2010]{Primary 22E46; Secondary 17B08.}

\keywords{symmetric $R$-space, degenerate principal series, generalized principal series, composition series, minimal representation, associated variety, nilpotent orbit}

\maketitle

\section*{Introduction}

Let $X=G/P$ be an irreducible symmetric $R$-space, i.e. $G$ is a simple real non-compact Lie group and $P$ a maximal parabolic subgroup with abelian nilradical. Then $X$ is at the same time a compact symmetric space $X=K/M$ where $K\subseteq G$ is a maximal compact subgroup of $G$ and $M=K\cap P$. Write $P=M_PA_PN$ for the Langlands decomposition of $P$ and let $\fraka_P$ be the Lie algebra of $A_P$. On $X$ we consider the degenerate principal series representations (normalized smooth parabolic induction)
\begin{equation*}
 I(\nu) := \Ind_P^G(\1\otimes e^{\nu}\otimes\1) \subseteq C^\infty(X)
\end{equation*}
of parameter $\nu\in(\fraka_P)_\CC^*$. In this paper we are concerned with the structure of these representations. Natural questions in this framework are:
\begin{enumerate}[(1)]
\item For which $\nu\in(\fraka_P)_\CC^*$ is $I(\nu)$ irreducible?
\item What is the composition series of a reducible $I(\nu)$?
\item Which irreducible constituents are unitary?
\end{enumerate}

A key role in the study of these questions is played by intertwining operators. In particular, a $G$-invariant subspace $W\subseteq I(\nu)$ is unitary if and only if there exists an intertwining operator $T:W\to I(-\overline{\nu})$ with strictly positive eigenvalues. In this case the inner product on $W$ given by
\begin{align*}
 W\times W\to\CC,\,(f,g)\mapsto\int_X{f(x)\overline{Tg(x)}\td x}
\end{align*}
is $G$-invariant.\\

Under the assumption that $P$ is conjugate to its opposite parabolic subgroup $\overline{P}$ the questions (1), (2) and (3) have been completely answered by Johnson~\cite{Joh92}, \O rsted--Zhang~\cite{OZ95}, Sahi~\cite{Sah93,Sah95} and Zhang~\cite{Zha95}. This case occurs exactly when there exists a Weyl group element for the restricted root system $\Delta$ of $\frakg$ (with respect to a maximally non-compact Cartan) which acts on $\fraka_P$ by $-\1$. Hence the classical Knapp--Stein theory provides standard intertwining operators $A(\nu):I(\nu)\to I(-\nu)$ with meromorphic dependence in the parameter $\nu\in(\fraka_P)_\CC^*$. In the above mentioned papers these Knapp--Stein intertwiners and their residues are used to treat the above questions.

In this paper we completely answer the above questions for the case where $P$ is not conjugate to $\overline{P}$ (see Tables~\ref{tb:ClassificationNonCompactGroups} and \ref{tb:ClassificationCompactGroups} for a classification). Up to covering the possible groups $G$, sorted by the restricted root system $\Delta$ of $\frakg$ with respect to a maximally non-compact Cartan, and the corresponding Lie algebras $\frakl$ of the Levi part $L=M_PA_P$ of $P=LN$ are
\begin{itemize}
\item Type $A$
\begin{itemize}
\item $G=\SL(r+s,\KK)$, $s>r\geq1$, with $\frakl=\fraks(\gl(r,\KK)\oplus\gl(s,\KK))$ and $\KK\in\{\RR,\CC,\HH\}$,
\item $G=\SL(3,\OO)=E_{6(-26)}$ with $\frakl=\so(1,9)\oplus\RR$,
\end{itemize}
\item Type $D$
\begin{itemize}
\item $G=\SO_0(2r+1,2r+1)$, $r>1$, with $\frakl=\gl(2r+1,\RR)$,
\item $G=\SO(4r+2,\CC)$, $r>1$, with $\frakl=\gl(2r+1,\CC)$,
\end{itemize}
\item Type $E_6$
\begin{itemize}
\item $G=E_{6(6)}$ with $\frakl=\so(5,5)\oplus\RR$,
\item $G=E_6(\CC)$ with $\frakl=\so(10,\CC)\oplus\CC$.
\end{itemize}
\end{itemize}
Let $R$ denote the set of restricted roots in $\Delta$ which contribute to the Lie algebra of $P$. Then $R$ is a closed subsystem $R$ of $\Delta$ which contains all positive roots. In all cases above there exists no element $w$ in the Weyl group of $\Delta$ with $wR=-R$. (Note that this can only occur for a root system of type $A_n$, $D_{2n+1}$ and $E_6$.) Hence the classical Knapp--Stein theory does not yield intertwiners $I(\nu)\to I(-\nu)$. Instead we employ the method of the ``spectrum generating operator'' by Branson--\'{O}lafsson--\O rsted \cite{BOO96}. Initially this method was invented to determine the eigenvalues of a given intertwining operator. However, it also provides a technique to construct intertwiners purely in terms of their eigenvalues. We use this technique to find \textit{non-standard} intertwining operators for all unitarizable subrepresentations of $I(\nu)$ and answer questions (1), (2) and (3) (see Theorem~\ref{thm:Reducibility} for the reducibility question, Theorem~\ref{thm:CompositionSeries} for the composition series and Theorems~\ref{thm:UnitarySubrepsSL}, \ref{thm:UnitarySubrepsE6} and \ref{thm:UnitarySubrepsSO} for the unitarity question). It is noteworthy that non-trivial unitarizable constituents other than the unitary principal series (i.e.\ $\nu\in i(\fraka_P)^*$) only occur for the restricted root system $\Delta$ of type $D$ or $E_6$. Here in particular the study of the unitarity question is delicate. To successfully apply the method of the ``spectrum generating operator'' we need to determine whether certain coefficients in the expansion of the product of two spherical functions on $X$ vanish (see Lemma~\ref{lem:CmConditionForUnitarity}). This is done using a result by Vretare~\cite{Vre76} and a combinatorial calculation carried out in Appendix~\ref{sec:CmFormula}.

We further study the associated varieties of the unitary constituents and find some representations with associated variety equal to the smallest nilpotent $K_\CC$-orbit in $\frakp_\CC^*$ for the groups $\SL(1+s,\KK)$ ($\KK\in\{\RR,\CC,\HH\}$), $\SO(2r+1,2r+1)$, $\SO(4r+2,\CC)$, $E_{6(-26)}$, $E_{6(6)}$ and $E_6(\CC)$ (see Section~\ref{sec:AssocVar}). These representations are all spherical, have one-dimensional asymptotic $K$-support and their Gelfand--Kirillov dimension attains its minimum among all infinite-dimensional unitary representations. $L^2$-models of these small representations in the spirit of Sahi~\cite{Sah92}, Dvorsky--Sahi~\cite{DS99,DS03}, Kobayashi--{\O}rsted~\cite{KO03c}, Barchini--Sepanski--Zierau~\cite{BSZ06} and Hilgert--Kobayashi--M.~\cite{HKM} will be the subject of a subsequent paper.\\

For the special linear groups $G=\SL(r+s,\KK)$, $\KK\in\{\RR,\CC,\HH\}$, with maximal parabolic subgroups $P=S(\GL(r,\KK)\times\GL(s,\KK))\ltimes\RR^{r\times s}$ our results on reducibility, composition series and unitarity are not new, but were earlier obtained by Howe--Lee~\cite{HL99} and Lee~\cite{Lee07} (see also \cite{DZ97,Pas99,vDM99} for the cases $r=1$). For $G=\SO_0(2r+1,2r+1)$ Johnson~\cite{Joh90} found the points of reducibility and the composition series. Although he did not explicitly answer the unitarity question it can be read off from his results. All remaining cases, in particular the exceptional groups, seem to not have been treated yet. Further work on degenerate principal series associated to maximal parabolic subgroups with possibly non-abelian nilradical can be found in \cite{Cla12,Fuj01,Gro71,HT93,Joh76,JW77,LL02,LL03}.

It is also worth mentioning the recent work of Clerc~\cite{Cle12} and \'{O}lafsson--Pasquale~\cite{OP12} on Knapp--Stein intertwining operators between degenerate principal series on symmetric $R$-spaces corresponding to different parabolics.\\

\noindent

\emph{Acknowledgement:} We would like to thank B.\ {\O}rsted for helpful discussions and for bringing the method of the ``spectrum generating operator'' to our attention.

\section{Preliminaries}

\subsection{Symmetric $R$-spaces and root data}\label{sec:RootData}
Let $G$ be a simple real non-compact Lie group with maximal parabolic subgroup $P\subseteq G$ whose nilradical is abelian. Then, by definition $X:=G/P$ is an irreducible symmetric $R$-space. The geometry and structure theory of symmetric $R$-spaces have first been studied by Nagano \cite{N65} and Takeuchi \cite{T65}. In the following, we recall some details. Let $P = LN$ denote the Levi decomposition of $P$, and let $K\subseteq G$ be a maximal compact subgroup such that $M:=L\cap K$ is maximal compact in $L$. Since $G=KP$, we have a natural identification $X=K/M$. Let $\frakg$, $\frakl$ and $\frakn$ be the Lie algebras corresponding to $G$, $L$ and $N$, and let $\theta$ denote the Cartan involution corresponding to $K\subseteq G$, resp.\ $\frakk\subseteq\frakg$. Then, there is a (unique) grading element $Z_0\in\frakz(\frakl)$ in the center of $\frakl$ satisfying $\theta Z_0 = -Z_0$ and such that $\frakg$ decomposes under the adjoint action of $Z_0$ into
\begin{align}\label{eq:grading}
	\frakg = \overline\frakn\oplus\frakl\oplus\frakn
\end{align}
with eigenvalues $-1,0,1$. Here $\frakl\oplus\frakn$ is the Lie algebra of $P$. We note that $\overline\frakn = \theta\frakn$. The involution $\sigma:=\Ad\exp(\pi\,i Z_0)$ on $\frakg_\CC$ leaves $\frakg$ invariant and acts on $\overline\frakn\oplus\frakn$ by $-1$, and on $\frakl$ by $1$. This induces a non-trivial involution (also denoted by $\sigma$) on $G$ satisfying $\sigma(K) = K$ such that $(K,M,\sigma)$ is a compact symmetric pair. Therefore, $X$ is a compact Riemannian symmetric space. We further remark that $\frakn$ naturally carries the structure of a simple real Jordan triple system (see \cite{Lo85}). In fact, this establishes a one-to-one correspondence between irreducible symmetric $R$-spaces and simple real Jordan triple systems. However, we will not use the Jordan triple structure in this paper.

The Cartan decomposition of $\frakg$ with respect to $\theta$ is denoted by
\begin{align}\label{eq:CartanDecomp}
	\frakg = \frakk\oplus\frakp.
\end{align}
Since $\theta Z_0 = -Z_0$, the decompositions \eqref{eq:grading} and \eqref{eq:CartanDecomp} are compatible, so
\[
	\frakg = \frakm\oplus\frakn_\frakk\oplus\frakl_\frakp\oplus\frakn_\frakp,
\]
where $\frakm=\frakl\cap\frakk$ is the Lie algebra of $M\subseteq L$, $\frakl_\frakp :=\frakl\cap\frakp$, and
\begin{align}\label{eq:nIntersection}
	\begin{aligned}
 		\frakn_\frakk &:= (\overline\frakn\oplus\frakn)\cap\frakk = \Set{X+\theta X}{X\in\frakn},\\
 		\frakn_\frakp &:= (\overline\frakn\oplus\frakn)\cap\frakp = \Set{X-\theta X}{X\in\frakn}.
 	\end{aligned}
\end{align}

Let $\fraka\subseteq\frakp$ be a maximal abelian subalgebra containing $Z_0$, and let $\Delta=\Delta(\frakg,\fraka)$ be the corresponding (restricted) root system. Then $\Delta$ splits into $\Delta = \Delta_{-1}\cup\Delta_0\cup\Delta_1$, where $\Delta_k = \Set{\lambda\in\Delta}{\lambda(Z_0) = k}$. We choose a positive system $\Delta^+\subseteq\Delta$ such that $\Delta_1\subseteq\Delta^+\subseteq\Delta_0\cup\Delta_1$, and let $\beta_1,\ldots,\beta_r$ be a maximal system of strongly orthogonal roots in $\Delta_1$ such that $\beta_1$ is the highest root in $\Delta$. The number $r$ is called the \textit{split rank} of the symmetric $R$-space $X$. For each $\beta_k$, we fix an $\mathfrak{sl}_2$-triple $(E_k,h_k,E_{-k})$, i.e.
\begin{align}\label{eq:sl2Triple}
	[h_k,E_k] = 2\,E_k,\quad
	[h_k,E_{-k}] = -2\,E_{-k},\quad
	[E_k,E_{-k}] = h_k,
\end{align}
satisfying $E_k\in\frakg_{\beta_k}$, $E_{-k} = -\theta E_k\in\frakg_{-\beta_k}$, $h_k\in\fraka$. Set $H_k := E_k-E_{-k}\in\frakn_\frakk$. Then
\begin{equation}
	\frakt := \bigoplus_{k=1}^r \RR H_k\subseteq\frakn_\frakk\label{eq:DefTorusT}
\end{equation}
is maximal abelian in $\frakn_\frakk$, since naturally $\frakn_\frakk\cong\frakn_\frakp$ and $(\frakm\oplus\frakn_\frakp,\frakm)$ is the non-compact symmetric dual to $(\frakk,\frakm)$ due to \cite{N65, T65}, see also \cite[Proposition~4.3]{K87a}. Moreover, if $\gamma_j\in\frakt_\CC^*$ is defined by $\gamma_j(iH_k):=2\delta_{jk}$, the (restricted) root system $\Delta(\frakg_\CC,\frakt_\CC)$ is given by either
\begin{align}
	&\Set{\frac{\pm\gamma_j\pm\gamma_k}{2}}{1\leq j,k\leq r}\setminus\{0\}
	&&\text{(type $C_r$)},\label{eq:gRootsTypeC}
\intertext{called the 'unital' case, or}
	&\Set{\frac{\pm\gamma_j\pm\gamma_k}{2},\pm\frac{\gamma_j}{2}}
			{1\leq j,k\leq r}\setminus\{0\}
	&&\text{(type $BC_r$)},\label{eq:gRootsTypeBC}
\end{align}
called the 'non-unital' case, see \cite[Proposition~2.2]{K87b}. The term 'unital' corresponds to the fact, that the natural Jordan triple structure on $\frakn$ comes from a unital Jordan algebra precisely if $\Delta(\frakg_\CC,\frakt_\CC)$ is of type $C_r$. As noted in the introduction, here we are interested in the non-unital case (see Tables~\ref{tb:ClassificationNonCompactGroups} and \ref{tb:ClassificationCompactGroups} for a classification).

\begin{table}[ht]
\begin{center}
\begin{tabular}{r||c|c|c|}
  \cline{2-4}
  & $\frakg$ & $\frakl$ & $\frakn$\\
  \hline
  \hline
\multicolumn{1}{|r||}{A.1} & $\sl(r+s,\RR)$, $s>r\geq1$ & $\sl(r,\RR)\oplus\sl(s,\RR)\oplus\RR$ &
$M(r\times s,\RR)$\\
\multicolumn{1}{|r||}{A.2} & $\sl(r+s,\CC)$, $s>r\geq1$ & $\sl(r,\CC)\oplus\sl(s,\CC)\oplus\CC$ &
$M(r\times s,\CC)$\\
\multicolumn{1}{|r||}{A.3} & $\sl(r+s,\HH)$, $s>r\geq1$ & $\sl(r,\HH)\oplus\sl(s,\HH)\oplus\RR$ &
$M(r\times s,\HH)$\\
  \multicolumn{1}{|r||}{A.4} & $\frake_{6(-26)}$ & $\so(1,9)\oplus\RR$ & $M(1\times2,\OO)$\\
  \hline
\multicolumn{1}{|r||}{D.1} & $\so(2r+1,2r+1)$, $r>1$ & $\sl(2r+1,\RR)\oplus\RR$ &
$\Skew(2r+1,\RR)$\\
\multicolumn{1}{|r||}{D.2} & $\so(4r+2,\CC)$, $r>1$ & $\sl(2r+1,\CC)\oplus\CC$ & $\Skew(2r+1,\CC)$\\  \hline
  \multicolumn{1}{|r||}{E.1} & $\frake_{6(6)}$ & $\so(5,5)\oplus\RR$ & $M(1\times2,\OO_s)$\\
  \multicolumn{1}{|r||}{E.2} & $\frake_6(\CC)$ & $\so(10,\CC)\oplus\CC$ & $M(1\times2,\OO_\CC)$\\
  \hline
\end{tabular}\\[2mm]
\caption{Non-unital irreducible symmetric $R$-spaces: non-compact Lie algebras\label{tb:ClassificationNonCompactGroups}}
\end{center}
\end{table}

\begin{table}[ht]
\begin{center}
\begin{tabular}{r||c|c|}
  \cline{2-3}
  & $\frakk$ & $\frakm$\\
  \hline
  \hline
  \multicolumn{1}{|r||}{A.1} & $\so(r+s)$ & $\so(r)\oplus\so(s)$\\
  \multicolumn{1}{|r||}{A.2} & $\fraks\fraku(r+s)$ & $\fraks(\fraku(r)\oplus\fraku(s))$\\
  \multicolumn{1}{|r||}{A.3} & $\sp(r+s)$ & $\sp(r)\oplus\sp(s)$\\
  \multicolumn{1}{|r||}{A.4} & $\frakf_4$ & $\so(9)$\\
  \hline
  \multicolumn{1}{|r||}{D.1} & $\so(2r+1)\oplus\so(2r+1)$ & $\so(2r+1)$\\
  \multicolumn{1}{|r||}{D.2} & $\so(4r+2)$ & $\fraku(2r+1)$\\
  \hline
  \multicolumn{1}{|r||}{E.1} & $\sp(4)$ & $\sp(2)\oplus\sp(2)$\\
  \multicolumn{1}{|r||}{E.2} & $\frake_6$ & $\so(10)\oplus\RR$\\
  \hline
\end{tabular}\\[2mm]
\caption{Non-unital irreducible symmetric $R$-spaces: compact Lie algebras\label{tb:ClassificationCompactGroups}}
\end{center}
\end{table}

From the classification we see that the root system $\Delta=\Delta(\frakg,\fraka)$ is either of type $A_n$, $D_{2n+1}$ or $E_6$. Note that in these cases there is no element in the Weyl group of $\Delta$ which acts on $\fraka$ by $-1$.

The intersection of the root spaces of $\Delta(\frakg_\CC,\frakt_\CC)$ with $\frakk_\CC\subseteq\frakg_\CC$ yields a root system $\Delta(\frakk_\CC,\frakt_\CC)\subseteq \frakt_\CC^*$. In the non-unital case the multiplicities
\[
	e:=\dim(\frakk_\CC)_{\pm\gamma_j},\
	d:=\dim(\frakk_\CC)_{\frac{\pm\gamma_j\pm\gamma_k}{2}},\
	b:=\dim(\frakk_\CC)_{\pm\frac{\gamma_j}{2}}
\]
are independent of the choice of $1\leq j<k\leq r$ and the signs (see \cite[\S\,11]{Loo77}) and we always have $b\neq0$. Therefore, the non-unital case splits into two possibilities:
\begin{align*}
	\Delta(\frakk_\CC,\frakt_\CC)\text{ is of type}
		\begin{cases}
			B_r & \text{if $b\neq 0$, $e=0$,}\\
			BC_r & \text{if $b\neq 0$, $e\neq0$.}
		\end{cases}
\end{align*}
For later purpose, it is convenient to introduce the \textit{genus} $p$ of $X$ defined by
\begin{equation}\label{eq:DefGenus}
	p := (e+1)+(r-1)\,d+\frac{b}{2}.
\end{equation}

Let $\kappa$ denote the Killing form of $\frakg$ and its bilinear complexification. Then
\[
	\kappa(H_k,H_\ell)
	=\sum_{\lambda\in\Delta(\frakg_\CC,\frakt_\CC)}
		\lambda(H_k)\lambda(H_\ell)\dim(\frakg_\CC)_\lambda.
\]
We thus obtain
\[
	\kappa(H_\ell,H_\ell) = -8\,\big(\tilde e+(r-1)\tfrac{\tilde d}{2}+\tfrac{\tilde b}{4}\big),
\]
where $\tilde e = \dim(\frakg_\CC)_{\pm\gamma_j}$, $\tilde d = \dim(\frakg_\CC)_{(\pm\gamma_j\pm\gamma_k)/2}$, $\tilde b = \dim(\frakg_\CC)_{\pm\gamma_j/2}$, independent of the choice of $1\leq j<k\leq r$ and signs. It is known that $\tilde e = e+1$, $\tilde d = 2d$ and $\tilde b = 2b$, and hence we note that
\begin{align}\label{eq:KillingP}
	p = -\frac{1}{8}\,\kappa(H_\ell,H_\ell),
\end{align}
independent of $1\leq\ell\leq r$.

We fix a lexicographic ordering of roots in $\Delta(\frakk_\CC,\frakt_\CC)$ by setting $\gamma_1>\gamma_2>\cdots>\gamma_r$. The corresponding set of positive roots is denoted by $\Delta^+(\frakk_\CC,\frakt_\CC)$. A straightforward calculation shows that the half sum of positive roots is given by
\begin{align}\label{eq:rhoformula}
	\rho^\frakk = \sum_{k=1}^r\rho_k\gamma_k
	\quad\text{with}\quad
	\begin{aligned}
		\rho_k &:= (r-k)\tfrac{d}{2} + \tfrac{e}{2} + \tfrac{b}{4} = \tfrac{p-1}{2}-(k-1)\tfrac{d}{2}.
	\end{aligned}
\end{align}

For the discussion of induced representations we set $\fraka_P:=\RR Z_0$ and $A_P:=\exp(\fraka_P)$ and let $P=M_PA_PN$ be the corresponding Langlands decomposition of $P$. Let $\gamma\in\fraka_P^*$ be the linear functional defined by $\gamma(Z_0) = 1$. Since \eqref{eq:grading} is the decomposition of $\frakg$ corresponding to the adjoint action of $\fraka_P$, it follows that $\Delta(\frakg,\fraka_P) = \{\pm\gamma\}$. The parabolic $P\subseteq G$ corresponds to the choice of $\gamma$ as the positive root, and hence
\begin{align}\label{eq:rho}
	\rho = \frac{n}{2}\,\gamma
\end{align}
is the half sum of positive roots where $n=\dim\frakn$.

\begin{table}[ht]
\begin{center}
\begin{tabular}{r||c|c|c|c|c|c|c|}
  \cline{2-8}
  & $\frakn$ & $n$ & $p$ & $r$ & $d$ & $e$ & $b$\\
  \hline
  \hline
  \multicolumn{1}{|r||}{A.1} & $M(1\times s,\RR)$, $s>1$ & $s$ & $(s+1)/2$ & $1$ & $0$ & $0$ & $s-1$\\
  \multicolumn{1}{|r||}{A.2} & $M(1\times s,\CC)$, $s>1$ & $2s$ & $s+1$ & $1$ & $0$ & $1$ & $2(s-1)$\\
  \multicolumn{1}{|r||}{A.3} & $M(1\times s,\HH)$, $s>1$ & $4s$ & $2(s+1)$ & $1$ & $0$ & $3$ & $4(s-1)$\\
  \multicolumn{1}{|r||}{A.1} & $M(r\times s,\RR)$, $s>r>1$ & $rs$ & $(r+s)/2$ & $r$ & $1$ & $0$ & $s-r$\\
  \multicolumn{1}{|r||}{A.2} & $M(r\times s,\CC)$, $s>r>1$ & $2rs$ & $r+s$ & $r$ & $2$ & $1$ & $2(s-r)$\\
  \multicolumn{1}{|r||}{A.3} & $M(r\times s,\HH)$, $s>r>1$ & $4rs$ & $2(r+s)$ & $r$ & $4$ & $3$ & $4(s-r)$\\
  \multicolumn{1}{|r||}{A.4} & $M(1\times2,\OO)$ & $16$ & $12$ & $1$ & $0$ & $7$ & $8$\\
  \hline
  \multicolumn{1}{|r||}{D.1} & $\Skew(2r+1,\RR)$, $r>1$ & $r(2r+1)$ & $2r$ & $r$ & $2$ & $0$ & $2$\\
  \multicolumn{1}{|r||}{D.2} & $\Skew(2r+1,\CC)$, $r>1$ & $2r(2r+1)$ & $4r$ & $r$ & $4$ & $1$ & $4$\\
  \hline
  \multicolumn{1}{|r||}{E.1} & $M(1\times2,\OO_s)$ & $16$ & $6$ & $2$ & $3$ & $0$ & $4$\\
  \multicolumn{1}{|r||}{E.2} & $M(1\times2,\OO_\CC)$ & $32$ & $12$ & $2$ & $6$ & $1$ & $8$\\
  \hline
\end{tabular}\\[2mm]
\caption{Non-unital irreducible symmetric $R$-spaces: structure constants\label{tb:ClassificationConstants}}
\end{center}
\end{table}

\subsection{$K$-types}

In this section, we describe the decomposition of $L^2(X)$ into $K$-irreducible subspaces by using the Cartan--Helgason theorem. For this we assume that $G$ has trivial center. Since $K$ is not necessarily simply-connected and $M$ is not necessarily connected we have to be careful in lifting Lie algebraic results to the group level.

For ${\bf m}=(m_1,\ldots, m_r)\in\ZZ^r$ let ${\bf m}\geq 0$ denote the condition
\[
	m_1\geq m_2\geq\cdots\geq m_r\geq 0.
\]
The main goal of this section is the following.

\begin{proposition}\label{prop:KtypeDecompositionL2K}
	Let $X$ be a non-unital symmetric $R$-space. Then the space $L^2(X)_{K\textup{-finite}}$ of $K$-finite vectors in 
	$L^2(X)$ decomposes into
	\begin{align*}
		 L^2(X)_{K\textup{-finite}} &= \bigoplus_{{\bf m}\geq0}{V^{\bf m}},
	\end{align*}
	where for ${\bf m}\geq 0$ the subrepresentation $V^{\bf m}$ is the irreducible 
	unitary $M$-spherical $K$-representation with highest weight
	$\lambda_{\bf m}:=\sum_{i=1}^r{m_i\gamma_i}$. In each $V^{\bf m}$ the space of
	$M$-spherical vectors is one-dimensional and there is a unique $M$-spherical
	vector $\phi_{\bf m}\in V^{\bf m}$ with $\phi_{\bf m}(x_0)=1$.
\end{proposition}

We first apply the Cartan--Helgason theorem on the Lie algebra level.

\begin{lemma}\label{lem:CartanHelgason}
	The highest weight of an irreducible $\frakk$-representation with an $\frakm$-spherical vector
	vanishes on the orthogonal complement of $\frakt$ in any maximal torus of $\frakk$ containing
	$\frakt$. The possible weights which give unitary irreducible $\frakm$-representations 
	are precisely given by
	\begin{align*}
		\Lambda^+_{\frakm}(\frakk) = 
	\begin{cases}
		\Set{\sum_{i=1}^r{t_i\gamma_i}}
				{t_i\in\tfrac{1}{2}\,\ZZ,t_i-t_j\in\ZZ,t_1\geq\ldots\geq t_r\geq0}
		& \mbox{in case $B$,}\\
		\Set{\sum_{i=1}^r{t_i\gamma_i}}
				{t_i\in\ZZ,t_1\geq\ldots\geq t_r\geq0}
		& \mbox{in case $BC$.}\end{cases}
	\end{align*}
	Further, in each irreducible $\frakm$-spherical $\frakk$-representation the space of 
	$\frakm$-spherical vectors is one-dimensional.
\end{lemma}
\begin{proof}
By the Cartan--Helgason theorem \cite[V\,\S4, Theorem~4.1]{Hel84} the set $\Lambda^+_{\frakm}(\frakk)$ of highest weights of $\frakm$-spherical irreducible $\frakk$-representations consists of all $\lambda\in\frakt_\CC^*$ such that
\begin{equation*}
 \frac{\kappa(\lambda,\alpha)}{\kappa(\alpha,\alpha)}\in\NN_0 \qquad \mbox{for all $\alpha\in\Delta^+(\frakk_\CC,\frakt_\CC)$,}
\end{equation*}
where we identify $\frakt_\CC^*\cong\frakt_\CC$ using $\kappa$. Writing $\lambda=\sum_{j=1}^rt_j\gamma_j$ with $t_j\in\RR$ we find
\begin{align*}
	\frac{\kappa(\lambda,\alpha)}{\kappa(\alpha,\alpha)} &=
	\begin{cases}
		t_j & \text{if $\alpha = \gamma_j$,}\\
		t_j\pm t_k & \text{if $\alpha = \tfrac{\gamma_j\pm\gamma_k}{2}$, $j<k$,}\\
		2t_j & \text{if $\alpha = \tfrac{\gamma_j}{2}$,}
	\end{cases}
\end{align*}
which implies the claimed statement for $\Lambda^+_\frakm(\frakk)$. Finally, \cite[VI Lemma~3.6]{Hel84} states that the space of $\frakm$-spherical vectors is one-dimensional.
\end{proof}

In order to lift this result to the group level, we need the following lemma.

\begin{lemma}\label{lem:UnitLatticeTorus}
	Let $T=\exp(\frakt)\subseteq K$.
	\begin{enumerate}[(1)]
		\item The exponential map $\frakt\to T$ has kernel $\sum_{j=1}^r{2\pi\ZZ H_j}$.
		\item $T$ meets every connected component of $M$, i.e.\ $M=(T\cap M)M_0$. Moreover
					\begin{equation}
						T\cap M = \exp\left(\sum_{j=1}^r{\pi\ZZ H_j}\right).\label{eq:UnitLatticeTorusKL}
					\end{equation}
	\end{enumerate}
\end{lemma}
\begin{proof}
Let $H:=\sum_{j=1}^r{t_jH_j}\in\frakt$ with $t_j\in\RR$, and consider the adjoint action of $\exp(H)$ on $\frakg_\CC$. Since $G$ is the adjoint group of $\frakg$, the element $H$ is in the kernel of $\exp$ if and only if $\Ad(\exp(H))=\id$. Since 
\[
	\Ad(\exp(H))X = e^{\sum_{j=1}^r t_j\lambda(H_j)}X\quad\text{for}\quad X\in(\frakg_\CC)_\lambda,
\]
and $\gamma_k(H_\ell) = -2i\delta_{k\ell}$, the decomposition \eqref{eq:gRootsTypeBC} immediately yields statement (1). For (2), we first note that $K=M_0TM_0$ and therefore $M=(T\cap M)M_0$, see \cite[Chapter V, Theorem 6.7]{Hel78}. Let $H=\sum_{j=1}^r t_jH_j$ with $t_j\in\RR$. Since $M=Z_K(\fraka_P)$ we have to show that $\Ad(\exp(H))Z_0=Z_0$ if and only if $t_j\in\pi\ZZ$ for all $j=1,\ldots,r$. Set $F_j:=E_j+E_{-j}$ with $E_{\pm j}$ defined in \eqref{eq:sl2Triple}. Since
$\ad(H_j)Z_0=-F_j$ and
\begin{align*}
 \ad(H_k)(F_j) &= 2\delta_{jk}\,h_k,\\
 \ad(H_k)(h_j) &= -2\delta_{jk}\,F_k,
\end{align*}
we find that
\begin{multline*}
 \Ad(\exp(H))Z_0 = \exp(\ad(H))Z_0
 	= Z_0 - \frac{1}{2}\,\sum_{j=1}^r\left(
 		\big(1-\cos(2t_j)\big)\,h_j+\sin(2t_j)\,F_j\right)
\end{multline*}
and the claim follows.\qedhere
\end{proof}

\begin{proof}[Proof of Proposition~\ref{prop:KtypeDecompositionL2K}]
First note that by Lemma~\ref{lem:UnitLatticeTorus}~(1) the lattice $\sum_{j=1}^r{\ZZ\gamma_j}$
consists of analytically integral functionals for $K$. Hence, for ${\bf m}\geq0$ the $\frakk$-action on $V^{\bf m}$ integrates to an action of $K$. Every $\frakm$-spherical
vector $\phi_{\bf m}\in V^{\bf m}$ is then automatically $M_0$-spherical. We claim that an
$M_0$-spherical vector $\phi_{\bf m}$ is also $M$-spherical. In fact, by
Lemma~\ref{lem:UnitLatticeTorus}~(2) it suffices to check that $k\cdot\phi_{\bf m}=\phi_{\bf m}$ for
$k\in T\cap M$. Write
\begin{equation*}
 \phi_{\bf m} = \int_{M_0}{(k_0\cdot v_{\bf m})\td k_0}
\end{equation*}
for some highest weight vector $v_{\bf m}\in V^{\bf m}$. Then by \eqref{eq:UnitLatticeTorusKL} we
have $k\cdot v_{\bf m}=v_{\bf m}$ and hence we obtain
\begin{align*}
 k\cdot\phi_{\bf m}
 &= \int_{M_0}{(kk_0\cdot v_{\bf m})\td k_0} = \int_{M_0}{(k_0k\cdot v_{\bf m})\td k_0}\\
 &= \int_{M_0}{(k_0\cdot v_{\bf m})\td k_0} = \phi_{\bf m}.
\end{align*}
It remains to show that the representations $V^{\bf m}$ comprise all irreducible $K$-represen\-tations
with an $M$-spherical vector. By Lemma~\ref{lem:CartanHelgason} it suffices to show that a
$K$-representation $V$ with highest weight $\sum_{j=1}^r{m_j\gamma_j}$, $m_1\geq\ldots\geq m_r\geq0$,
can only have an $M$-invariant vector if $m_j\in\ZZ$ for all $j=1,\ldots,r$. Let $\phi\neq0$
be an $M$-invariant vector and let $v$ be a highest weight vector. If $(\blank|\blank)$ denotes the
$K$-invariant inner product on $V$ then $(k\cdot v|\phi)=(v|\phi)\neq0$ for all $k\in M$. By Lemma~\ref{lem:UnitLatticeTorus}~(2)
the element $k:=\exp(\pi H_j)$ is contained in $M$ and $k\cdot v=e^{2\pi im_j}v$. Hence $m_j\in\ZZ$ for all $j=1,\ldots,r$ and the proof is complete.
\end{proof}

\section{Degenerate principal series and the spectrum generating operator}\label{sec:SpectrumGeneratingOperator}

For the rest of this article let $X$ be a non-unital symmetric $R$-space.

\subsection{Degenerate principal series}

Recall the element $Z_0\in\fraka_P$ from \eqref{eq:grading}. We identify $(\fraka_P)_\CC^*$ with $\CC$ by the map $\nu\mapsto\frac{p}{n}\nu(Z_0)$. Then $\rho$ from \eqref{eq:rho} corresponds to $\frac{p}{2}$. (The reason for this normalization is the simplicity of the formulas in Proposition~\ref{prop:ExplicitFormularhosInTermsOfOmega}.) We form the induced representation (normalized smooth parabolic induction)
\begin{multline*}
 I(\nu) := \Ind_P^G(\1\otimes e^\nu\otimes\1) = \{\varphi\in C^\infty(G):\varphi(gman)=a^{-\nu-\rho}\varphi(g)\\
 \forall\,g\in G,man\in M_PA_PN\}.
\end{multline*}
Denote by $\pi_\nu$ the corresponding action of $G$ on $I(\nu)$ by left-translation.

Since $G=KP$, restriction to $K$ defines an isomorphism $I(\nu)\to C^\infty(X)$. Denote by $(\rho_\nu,C^\infty(X))$ the corresponding representation of $G$, i.e.
\begin{align*}
 \rho_\nu(g)(\varphi|_K) &:= (\pi_\nu(g)\varphi)|_K & g\in G,\varphi\in I(\nu).
\end{align*}
The restriction of $\rho_\nu$ to $K$ is the left-regular representation on $C^\infty(X)$ and hence, by Proposition~\ref{prop:KtypeDecompositionL2K}, the $K$-finite vectors $C^\infty(X)_{K\textup{-finite}}$ decompose into
\begin{align*}
 C^\infty(X)_{K\textup{-finite}} &\cong \bigoplus_{{\bf m}\geq0}{V^{\bf m}},
\end{align*}
each $K$-type $V^{\bf m}$ appearing with multiplicity one. In what follows we will identify $V^{\bf m}$ with the corresponding subspace in $C^\infty(X)$ or $I(\nu)$ and it will be clear from the context which identification we use. Further note that 
\[
	\langle\varphi,\psi\rangle_{L^2(X)} = \int_K \varphi(k)\overline{\psi(k)} dk
\]
provides a sesquilinear form on $C^\infty(X)$ which is invariant under the representations $\rho_\nu\otimes\rho_{-\overline{\nu}}$, i.e.
\begin{align*}
 \langle\rho_\nu(g)\varphi,\rho_{-\overline{\nu}}(g)\psi\rangle_{L^2(X)} &= \langle\varphi,\psi\rangle_{L^2(X)}, & \varphi,\psi\in C^\infty(X).
\end{align*}

\subsection{The spectrum generating operator}

Let us recall some results from \cite{BOO96}. Denote by $\td\rho_\nu$ the differentiated representation of $\frakg$ on $C^\infty(X)$. Let $\langle\blank,\blank\rangle:=\tfrac{1}{2n}\kappa(\blank,\blank)$ be the renormalization of the Killing form of $\frakg$ such that $\langle Z_0,Z_0\rangle = 1$. On $\frakk$ the bilinear form $-\langle\blank,\blank\rangle$ is positive definite and we choose an orthonormal basis $(T_\alpha)_\alpha\subseteq\frakn_\frakk$ of the form $T_\alpha=N_\alpha+\theta N_\alpha$, $N_\alpha\in\frakn$. Consider the quadratic element
\begin{equation*}
 \calP := \sum_\alpha T_\alpha^2 \in\calU(\frakk).
\end{equation*}
Then by \cite[Lemma 2.2]{BOO96} the right regular action $\calR_\calP$ of $\calP$ leaves $C^\infty(X)$ invariant. This operator is called the ``spectrum generating operator''. In our case we can obviously express $\calP$ in terms of the Casimir elements $C_\frakk\in\calU(\frakk)$ and $C_\frakm\in\calU(\frakm)$ of $\frakk$ and $\frakm$, respectively, taken with respect to the inner product $\langle\blank,\blank\rangle$ (cf. \cite[Remark 2.4]{BOO96}):
\begin{equation*}
 \calP = C_\frakk-C_\frakm.
\end{equation*}
Since the right regular action of $C_\frakm$ on $C^\infty(X)$ is trivial and the right regular action of $C_\frakk$ on $C^\infty(X)$ coincides with the left-regular action we find that the spectrum generating operator $\calR_\calP$ agrees with the Laplacian $\td\rho_\nu(C_\frakk)$ of $X$ which is independent of $\nu$. As Casimir operator, $\td\rho_\nu(C_\frakk)$ acts on each $K$-type $V^{\bf m}$ by the eigenvalue
\begin{align*}
 \pi_{\bf m} &:= -\langle\lambda_{\bf m}+2\rho^\frakk,\lambda_{\bf m}\rangle,
\end{align*}
where we identify $\frakt_\CC^*\cong\frakt_\CC$ via the non-degenerate bilinear form $\langle\blank,\blank\rangle$. A short calculation using \eqref{eq:KillingP} shows that we have $\langle\gamma_j,\gamma_k\rangle=-\frac{n}{p}\delta_{jk}$ and hence
\begin{align}
 \pi_{\bf m} &= \frac{n}{p}\sum_{j=1}^r{(m_j+2\rho_j)m_j} = \frac{n}{p}\sum_{j=1}^r{\left(m_j^2+(p-1-(j-1)d)m_j\right)}.\label{eq:Pim}
\end{align}
Now define a function $\omega:\frakg_\CC\to C^\infty(X)$ by
\[
	\omega(T)(k):=\langle T,\Ad(k)Z_0\rangle.
\]
Since $Z_0\in\frakp_\CC$ and $\frakp_\CC$ is $\Ad(K)$-invariant, we clearly have $\omega|_{\frakk_\CC}=0$. By the invariance of the bilinear form $\langle\blank,\blank\rangle$ we further obtain that $\omega$ is $K$-equivariant. Let $m(\omega(T)):C^\infty(X)\to C^\infty(X)$ be the multiplication operator $\varphi\mapsto\omega(T)\varphi$. Then by \cite[equation (2.1)]{BOO96} we have
\begin{align*}
 (\td\rho_\nu-\td\rho_{\nu'})(T) &= \tfrac{n}{p}(\nu-\nu')m(\omega(T)), & T\in\frakg_\CC.
\end{align*}
For ${\bf m}\geq0$ let $P_{\bf m}:L^2(X)\to V^{\bf m}$ denote the orthogonal projection onto $V^{\bf m}\subseteq L^2(X)$. Then by \cite[Corollary 2.6]{BOO96} we further have for ${\bf m},{\bf n}\geq0$:
\begin{align}
 2P_{\bf n}\circ\td\rho_\nu(T)|_{V^{\bf m}} &= (\tfrac{2n}{p}\nu+\pi_{\bf n}-\pi_{\bf m})P_{\bf n}\circ m(\omega(T))|_{V^{\bf m}}, & T\in\frakp_\CC.\label{eq:GeneralExplicitExpressiondrhoInTermsOfOmega}
\end{align}
Consider the map
\[
	\Phi_{\bf m}:\frakp_\CC\otimes V^{\bf m}\to C^\infty(X),\
	T\otimes\varphi\mapsto m(\omega(T))\varphi\;.
\]
One readily checks that $\Phi_{\bf m}$ is $K$-equivariant, and hence there is a subset $S_{\bf m}\subseteq\Set{{\bf m}\in\ZZ^r}{{\bf m}\geq0}$ such that
\[
	\Phi_{\bf m}(\frakp_\CC\otimes V^{\bf m}) = \bigoplus_{{\bf n}\in S_{\bf m}} V^{\bf n}\;.
\]
To determine $S_{\bf m}$ explicitly we need the following lemma on the spherical vectors $\phi_{\bf m}\in V^{\bf m}$, ${\bf m}\geq 0$. For convenience we set $\phi_{\bf m}:=0$ for ${\bf m}\in\ZZ^r$ not satisfying ${\bf m}\geq 0$.

\begin{lemma}\label{lem:Vretare}
Let $X$ be non-unital and ${\bf m}\geq0$. Then
\begin{align*}
 \omega(Z_0)\phi_{\bf m} &= \sum_{k=1}^r{A({\bf m},k)\phi_{{\bf m}+e_k}}+\sum_{k=1}^r{B({\bf m},k)\phi_{{\bf m}-e_k}}+C({\bf m})\phi_{\bf m},
\end{align*}
where
\begin{align*}
 A({\bf m},k) ={}& \frac{p}{2n}\frac{(2(m_k+\rho_k)+\frac{b}{2}+1)(2(m_k+\rho_k)+\frac{b}{2}+e)}{2(m_k+\rho_k)(2(m_k+\rho_k)+1)}\\
 & \times\prod_{j\neq k}{\frac{((m_k+\rho_k)-(m_j+\rho_j)+\frac{d}{2})((m_k+\rho_k)+(m_j+\rho_j)+\frac{d}{2})}{((m_k+\rho_k)-(m_j+\rho_j))((m_k+\rho_k)+(m_j+\rho_j))}},\\
 B({\bf m},k) ={}& \frac{p}{2n}\frac{(2(m_k+\rho_k)-\frac{b}{2}-1)(2(m_k+\rho_k)-\frac{b}{2}-e)}{2(m_k+\rho_k)(2(m_k+\rho_k)-1)}\\
 & \times\prod_{j\neq k}{\frac{((m_k+\rho_k)+(m_j+\rho_j)-\frac{d}{2})((m_k+\rho_k)-(m_j+\rho_j)-\frac{d}{2})}{((m_k+\rho_k)+(m_j+\rho_j))((m_k+\rho_k)-(m_j+\rho_j))}},\\
 C({\bf m}) ={}& 1-\sum_{\substack{k=1\\{\bf m}+e_k\geq0}}^r{A({\bf m},k)}-\sum_{\substack{k=1\\{\bf m}-e_k\geq0}}^r{B({\bf m},k)}
\end{align*}
as rational functions in ${\bf m}$. In particular,
\begin{multicols}{2}
\begin{itemize}
\item $A({\bf m},k)\neq0$ for ${\bf m}+e_k\geq0$,
\item $B({\bf m},k)\neq0$ for ${\bf m}-e_k\geq0$ or $k=r,m_r=0,\rho_r=\frac{1}{2}$,
\item $A({\bf m},k)=0$ for ${\bf m}+e_k\ngeq0$,
\item $B({\bf m},k)=0$ for ${\bf m}-e_k\ngeq0$ with either $k\neq r$ or $k=r,\rho_r\neq\frac{1}{2}$
\end{itemize}
\end{multicols}
\end{lemma}

\begin{proof}
Note that $\frakp_\CC$ is an irreducible $\frakk$-module and the element $Z_0\in\frakp_\CC$ is $\frakm$-invariant. Moreover, $\gamma_1$ is the highest weight of $\frakp_\CC$, so $\frakp_\CC\cong V^{e_1}$. Now $\omega:\frakp_\CC\to C^\infty(X)$ is $\frakk$-equivariant and therefore $\omega(Z_0)$ is a scalar multiple of the spherical vector $\phi_{e_1}$. Since $\phi_{e_1}(eM)=1=\omega(Z_0)(eM)$ we even have $\omega(Z_0)=\phi_{e_1}$.
Using \cite[Theorem 4.8]{Vre76} we obtain that the claimed expansion of the product $\phi_{e_1}\phi_{\bf m}$ holds with
\begin{align*}
 A({\bf m},k) &= \frac{c(-i(\gamma_1+\rho))c(-iS_k^+({\bf m}+\rho))}{c(-i\rho)c(-i(S_k^+({\bf m}+\rho)+\gamma_1))},\\
 B({\bf m},k) &= \frac{c(-i(\gamma_1+\rho))c(-iS_k^-({\bf m}+\rho))}{c(-i\rho)c(-i(S_k^-({\bf m}+\rho)+\gamma_1))},
\end{align*}
where $c(\lambda)$ denotes the $c$-function of the root system $\Sigma=\Delta(\frakk_\CC,\frakt_\CC)$ and
\begin{align*}
 S_k^\pm(\lambda_1,\ldots,\lambda_r) &= (\pm\lambda_k,\lambda_2,\ldots,\lambda_{k-1},\lambda_1,\lambda_{k+1},\ldots,\lambda_r).
\end{align*}
By the Gindikin--Karpelevich formula
\begin{align*}
 c(\lambda) &= c_0\prod_{\alpha\in\Sigma_0}{\frac{2^{-i\frac{\langle\lambda,\alpha\rangle}{\langle\alpha,\alpha\rangle}}\Gamma(i\frac{\langle\lambda,\alpha\rangle}{\langle\alpha,\alpha\rangle})}{\Gamma(\frac{1}{2}(\frac{1}{2}\mu_\alpha+1+i\frac{\langle\lambda,\alpha\rangle}{\langle\alpha,\alpha\rangle}))\Gamma(\frac{1}{2}(\frac{1}{2}\mu_\alpha+\mu_{2\alpha}+i\frac{\langle\lambda,\alpha\rangle}{\langle\alpha,\alpha\rangle}))}}
\end{align*}
with some constant $c_0\neq0$ and $\Sigma_0:=\{\alpha\in\Sigma^+:\frac{\alpha}{2}\notin\Sigma\}$. Since $\Sigma$ is of type $B$ or $BC$ we have
\begin{equation*}
 \Sigma_0 = \set{\tfrac{\gamma_j}{2}}{1\leq j\leq r}
 	\cup\set{\tfrac{\gamma_j\pm\gamma_k}{2}}{1\leq j<k\leq r}
\end{equation*} 
with multiplicities
\begin{align*}
 \mu_{\frac{\gamma_j}{2}} &= b, & 
 \mu_{\frac{\gamma_j-\gamma_k}{2}} &= \mu_{\frac{\gamma_j+\gamma_k}{2}} = d, & 
 \mu_{\gamma_j} &= e.
\end{align*}
Using the identity $\Gamma(z)\Gamma(z+\frac{1}{2})=\sqrt{\pi}2^{1-2z}\Gamma(2z)$ for the gamma function we obtain the following general expression for the $c$-function at $\lambda=\sum_{j=1}^r{\lambda_j\gamma_j}$:
\begin{multline*}
 c(\lambda) = c_0'\prod_{j=1}^r{\frac{2^{-2i\lambda_j}\Gamma(2i\lambda_j)}{\Gamma(\frac{1}{2}(\frac{b}{2}+1+2i\lambda_j))\Gamma(\frac{1}{2}(\frac{b}{2}+e+2i\lambda_j))}}\\
 \times\prod_{1\leq j<k\leq r}{\frac{\Gamma(i(\lambda_j-\lambda_k))\Gamma(i(\lambda_j+\lambda_k))}{\Gamma(\frac{d}{2}+i(\lambda_j-\lambda_k))\Gamma(\frac{d}{2}+i(\lambda_j+\lambda_k))}}
\end{multline*}
with some constant $c_0'\neq0$. With \eqref{eq:rhoformula} we find
\begin{align*}
 & \frac{c(-i(\gamma_1+\rho))}{c(-i\rho)}\\
 ={}& \frac{(2\rho_1)(2\rho_1+1)}{(2\rho_1+\frac{b}{2}+1)(2\rho_1+\frac{b}{2}+e)}\prod_{j>1}{\frac{(\rho_1-\rho_j)(\rho_1+\rho_j)}{(\rho_1-\rho_j+\frac{d}{2})(\rho_1+\rho_j+\frac{d}{2})}}\\
 ={}& \frac{(2\rho_1+1)}{2r(2\rho_1+\frac{b}{2}+1)} = \frac{p}{2n}
\intertext{and}
 & \frac{c(-iS_k^\pm({\bf m}+\rho))}{c(-i(S_k^\pm({\bf m}+\rho)+\gamma_1))}\\
 ={}& \frac{(\pm2(m_k+\rho_k)+\frac{b}{2}+1)(\pm2(m_k+\rho_k)+\frac{b}{2}+e)}{\pm2(m_k+\rho_k)(\pm2(m_k+\rho_k)+1)}\\
 & \vspace{3cm}\times\prod_{j\neq k}{\frac{(\pm(m_k+\rho_k)-(m_j+\rho_j)+\frac{d}{2})(\pm(m_k+\rho_k)+(m_j+\rho_j)+\frac{d}{2})}{(\pm(m_k+\rho_k)-(m_j+\rho_j))(\pm(m_k+\rho_k)+(m_j+\rho_j))}}
\end{align*}
and thus the formulas above follow. We now prove the remaining claims for $B({\bf m},k)$, the corresponding results for $A({\bf m},k)$ are shown similarly. First note that we can rewrite $B({\bf m},k)$ as follows:
\begin{align*}
 & \frac{p}{2n}\frac{(2(m_k+\rho_k)-\frac{b}{2}-1)(2(m_k+\rho_k)-\frac{b}{2}-e)}{2(m_k+\rho_k)(2(m_k+\rho_k)-1)}\\
 & \times\prod_{j<k}{\frac{(m_j+\rho_j)-(m_k+\rho_k)+\frac{d}{2}}{(m_j+\rho_j)-(m_k+\rho_k)}}\prod_{j>k}{\frac{(m_k+\rho_k)-(m_j+\rho_j)-\frac{d}{2}}{(m_k+\rho_k)-(m_j+\rho_j)}}\\
 & \times\prod_{j\neq k}{\frac{(m_k+\rho_k)+(m_j+\rho_j)-\frac{d}{2})}{(m_k+\rho_k)+(m_j+\rho_j)}}.
\end{align*}
Recall from \eqref{eq:rhoformula} that $\rho_j=(r-j)\tfrac{d}{2}+\tfrac{e}{2}+\tfrac{b}{4}\geq\tfrac{b}{4}\geq\frac{1}{4}>0$ for every $j=1,\ldots,r$. Further, in the case $r>1$ we have $\rho_j-\rho_\ell=(\ell-j)\frac{d}{2}\geq\frac{d}{2}>0$ for $j<\ell$. Therefore, for ${\bf m}\geq0$ the denominator of $B({\bf m},k)$ is always non-zero if $2(m_k+\rho_k)-1\neq0$. Since $m_k\in\NN_0$ we can only have $2(m_k+\rho_k)-1=0$ for $m_k=0,\rho_k=\frac{1}{2}$. Further, since $\rho_r=\frac{e}{2}+\frac{b}{4}\geq\frac{1}{4}$ and $\rho_k\geq\frac{d}{2}+\frac{b}{4}\geq\frac{3}{4}$ for $k\neq r$ we can only have $\rho_k=\frac{1}{2}$ for $k=r$ and $b=2,e=0$. In this case, i.e. $k=r,m_r=0,\rho_r=\frac{1}{2}$, the factor $2(m_r+\rho_r)-\frac{b}{2}-e$ in the numerator cancels with the factor $2(m_r+\rho_r)-1$ in the denominator. It is further easy to see that all other factors in the numerator are non-zero, so in this case $B({\bf m},r)\neq0$. If we suppose that ${\bf m},{\bf m}-e_k\geq0$, i.e.
\begin{equation*}
 m_1\geq\ldots\geq m_k>m_k-1\geq m_{k+1}\geq\cdots\geq m_r\geq0,
\end{equation*}
using the convention $m_{r+1}=0$, then all terms in the numerator and denominator are non-zero and hence again $B({\bf m},k)\neq0$. Finally, suppose that ${\bf m}-e_k\ngeq0$, i.e. $m_k=m_{k+1}$. If $k\neq r$ then the factor $(m_k+\rho_k)-(m_j+\rho_j)-\frac{d}{2}$ in the numerator vanishes for $j=k+1$. If $k=r,\rho_r\neq\frac{1}{2}$ then $m_r=0$ and the factor $2(m_k+\rho_k)-\frac{b}{2}-e$ in the numerator vanishes. This completes the proof.
\end{proof}

To determine whether ${\bf m}\in S_{\bf m}$ we need another result on the tensor product of representations occuring in $C^\infty(X)$.

\begin{lemma}[{see \cite[Proposition 3.1]{Joh92}}]\label{lem:MultOnXForSphericalVectors}
Let $V_1,V_2,W\subseteq C^\infty(X)$ be irreducible representations of $K$ such that $W\subseteq V_1\cdot V_2$, the subspace of $C^\infty(X)$ spanned by products of functions in $V_1$ and functions in $V_2$. Denote by $P:L^2(X)\to W$ the orthogonal projection onto $W$. Then the map
$$ \Phi:V_1\otimes V_2\to W, \quad (\varphi_1,\varphi_2)\mapsto P(\varphi_1\cdot\varphi_2) $$
satisfies $\Phi(V_1^M\otimes V_2^M)\neq0$.
\end{lemma}

\begin{proof}
First note that for an irreducible $K$-representation $V\subseteq C^\infty(X)$ we can identify its contragredient $V^*$ with the subrepresentation $\overline{V}\subseteq C^\infty(X)$ given by complex conjugates of functions in $V$. Then $W\subseteq V_1\cdot V_2$ clearly implies $\overline{W}\subseteq\overline{V}_1\cdot\overline{V}_2$ and hence the assumptions are also satisfied for the contragredient representations $V_1^*$, $V_2^*$ and $W^*$.\par
By the Peter--Weyl Theorem we have
$$ L^2(K)_{(K\times K)\textup{-finite}} \cong \bigoplus_{\pi\in\widehat{K}} \pi\boxtimes\pi^* $$
as $K\times K$-representations, identifying $V_1\boxtimes V_1^*$, $V_2\boxtimes V_2^*$ and $W\boxtimes W^*$ naturally with subspaces of $L^2(K)$. Viewing functions on $X=K/M$ as functions on $K$ which are right-invariant under $M$ we have the natural isomorphism
$$ L^2(X)_{K\textup{-finite}} \cong \bigoplus_{\pi\in\widehat{K}} \pi\otimes(\pi^*)^M. $$
Since $V_1,V_2,W\subseteq C^\infty(X)$ there are non-zero $M$-invariant functionals $\mu_1\in(V_1^*)^M$, $\mu_2\in(V_2^*)^M$ and $\nu\in(W^*)^M$ such that the maps
\begin{align*}
 V_1 &\to C^\infty(X), & \phi_1 &\mapsto(kM\mapsto\langle k\cdot\mu_1,\phi_1\rangle),\\
 V_2 &\to C^\infty(X), & \phi_2 &\mapsto(kM\mapsto\langle k\cdot\mu_2,\phi_2\rangle),\\
 W &\to C^\infty(X), & \psi &\mapsto(kM\mapsto\langle k\cdot\nu,\psi\rangle)
\end{align*}
are the identity on $V_1$, $V_2$ and $W$, respectively. Consider the map
$$ \Psi:(V_1\boxtimes V_1^*)\otimes(V_2\boxtimes V_2^*)\to W\boxtimes W^* $$
given by multiplication of two functions in $V_1\boxtimes V_1^*$ and $V_2\boxtimes V_2^*$ on $K$ followed by orthogonal projection onto the component $W\boxtimes W^*\subseteq L^2(K)$. Since $\Psi$ is $(K\times K)$-equivariant, $W$ is irreducible and
$$ (V_1\boxtimes V_1^*)\otimes(V_2\boxtimes V_2^*) \cong (V_1\otimes V_2)\boxtimes(V_1^*\otimes V_2^*) $$
the map $\Psi$ is given by the outer tensor product of a map $\Psi_1:V_1\otimes V_2\to W$ and a map $\Psi_2:V_1^*\otimes V_2^*\to W^*$, i.e.
$$ \Psi = \Psi_1\boxtimes\Psi_2. $$
We claim that $\Psi_1$ is up to a constant equal to the map $\Phi$. In fact, plugging in the non-zero element $0\neq\mu=\mu_1\otimes\mu_2\in(V_1^*)^M\otimes(V_2^*)^M$ we find
$$ \Psi_1(\phi\otimes\psi)\boxtimes\Psi_2(\mu) = \Psi((\phi\otimes\psi)\boxtimes\mu) = \Psi((\phi\boxtimes\mu_1)\otimes(\psi\boxtimes\mu_2)). $$
The right hand side is the orthogonal projection of the product of the $M$-invariant functions $\phi\boxtimes\mu_1$ and $\psi\boxtimes\mu_2$ on $K$ onto $W\boxtimes W^*$ and hence equal to $\Phi(\phi\otimes\psi)\boxtimes\nu$. We find that $\Psi_1(\phi\otimes\psi)=C\cdot\Phi(\phi\otimes\psi)$ and $\Psi_2(\mu)=C^{-1}\cdot\nu$ for some constant $C\neq0$. In particular, $\Psi_2((V_1^*)^M\otimes(V_2^*)^M)\neq0$. Repeating the same argument for the contragredient representations we find that also $\Psi_1(V_1^M\otimes V_2^M)\neq0$ and the proof is complete.
\end{proof}

\begin{corollary}\label{cor:Sm}
For each ${\bf m}\geq0$ we have
\begin{equation*}
 \Set{{\bf m}\pm e_j\geq0}{1\leq j\leq r}\subseteq S_{\bf m}\subseteq
 \Set{{\bf m},{\bf m}\pm e_j\geq0}{1\leq j\leq r}.
\end{equation*}
Further, ${\bf m}\in S_{\bf m}$ if and only if $C({\bf m})\neq0$.
\end{corollary}

\begin{proof}
For every ${\bf n}\in S_{\bf m}$ the $K$-module $V^{\bf n}$ appears in the tensor product $\frakp_\CC\otimes V^{\bf m}$. By \cite[Proposition 2.1]{Joh88} the highest weight of a subrepresentation of $\frakp_\CC\otimes V^{\bf m}$ is the sum of the highest weight of $V^{\bf m}$ and a weight of $\frakp_\CC$. Since the restricted weights of $\frakp_\CC$ are contained in the set $\set{\frac{\pm\gamma_j\pm\gamma_k}{2}}{1\leq j,k\leq r}\cup\set{\pm\frac{\gamma_j}{2}}{1\leq j\leq r}$ the only possible spherical representations $V^{\bf n}$ contained in $\frakp_\CC\otimes V^{\bf m}$ satisfy by Proposition~\ref{prop:KtypeDecompositionL2K} either ${\bf n}={\bf m}$ or ${\bf n}={\bf m}\pm e_j$ for some $1\leq j\leq r$. Hence the second inclusion follows. The first inclusion is an immediate consequence of Lemma~\ref{lem:Vretare}.\\
Now if $C({\bf m})\neq0$ then clearly ${\bf m}\in S_{\bf m}$. It remains to show that ${\bf m}\in S_{\bf m}$ implies $C({\bf m})\neq0$ which is simply the statement of Lemma~\ref{lem:MultOnXForSphericalVectors} for $V_1=\omega(\frakp_\CC)=V^{\gamma_1}$ and $V_2=W=V^{\bf m}$.
\end{proof}

\begin{proposition}\label{prop:ExplicitFormularhosInTermsOfOmega}
For ${\bf m}\geq0$ we have
\begin{align*}
 \td\rho_\nu(\frakg_\CC)V^{\bf m} \subseteq \bigoplus_{{\bf n}\in S_{\bf m}}V^{\bf n}.
\end{align*}
More precisely, we have for $T\in\frakp_\CC$
\begin{align*}
 2P_{{\bf m}+e_j}\circ\td\rho_\nu(T)|_{V^{\bf m}} &= \tfrac{2n}{p}(\nu+{\bf m}^+(j))P_{{\bf m}+e_j}\circ m(\omega(T))|_{V^{\bf m}},\\
 2P_{{\bf m}-e_j}\circ\td\rho_\nu(T)|_{V^{\bf m}} &= \tfrac{2n}{p}(\nu-{\bf m}^-(j))P_{{\bf m}-e_j}\circ m(\omega(T))|_{V^{\bf m}},\\
 2P_{\bf m}\circ\td\rho_\nu(T)|_{V^{\bf m}} &= \tfrac{2n}{p}\nu P_{\bf m}\circ m(\omega(T))|_{V^{\bf m}},
\end{align*}
where
\begin{align*}
 {\bf m}^+(j) &:= m_j+\frac{p}{2}-(j-1)\frac{d}{2}, & {\bf m}^-(j) &:=(m_j-1)+\frac{p}{2}-(j-1)\frac{d}{2}.
\end{align*}
\end{proposition}

\begin{proof}
The first claim follows directly from Corollary~\ref{cor:Sm}. For the explicit formulas we use \eqref{eq:Pim} and \eqref{eq:GeneralExplicitExpressiondrhoInTermsOfOmega} and observe that
\begin{align}
 \pi_{{\bf m}+e_j} -\pi_{\bf m} &= \frac{2n}{p}\left(m_j+\frac{p}{2}-(j-1)\frac{d}{2}\right).\label{eq:DifferencePim}
\end{align}
This completes the proof.
\end{proof}

\section{Reducibility and composition series}\label{sec:CompositionSeries}

We now determine all points of reducibility for $I(\nu)$ and find the complete composition series in these cases.

\subsection{Reducibility}

Recall the constants $A({\bf m},j)$, $B({\bf m},j)$, $C({\bf m})$ from Lemma~\ref{lem:Vretare} and ${\bf m}^+(j)$, ${\bf m}^-(j)$ from Proposition~\ref{prop:ExplicitFormularhosInTermsOfOmega}.

\begin{corollary}\label{cor:ActionOfH0OnSphericalVectors}
For every ${\bf m}\geq0$ and $s\in\RR$ we have
\begin{multline}
 \td\rho_\nu(Z_0)\phi_{\bf m} = \frac{2n}{p}\sum_{k=1}^r{\Big(A({\bf m},k)(\nu+{\bf m}^+(k))\phi_{{\bf m}+e_k}}\\
 +B({\bf m},k)(\nu-{\bf m}^-(k))\phi_{{\bf m}-e_k}\Big)+\frac{2n}{p}C({\bf m})\nu\phi_{\bf m}.\label{eq:ActionOfH0OnSphericalVectors}
\end{multline}
\end{corollary}

\begin{proof}
This follows directly from Lemma~\ref{lem:Vretare} and Proposition~\ref{prop:ExplicitFormularhosInTermsOfOmega}.
\end{proof}

\begin{theorem}\label{thm:Reducibility}
Let $X$ be a non-unital symmetric $R$-space. Then the principal series representation $I(\nu)$ is reducible if and only if either
\begin{align*}
 \nu\in-\NN_0-\frac{p}{2}+(j-1)\frac{d}{2} && \mbox{or} && \nu\in\NN_0+\frac{p}{2}-(j-1)\frac{d}{2}
\end{align*}
for some $1\leq j\leq r$. In particular, $I(\nu)$ is irreducible for all $\nu\in i\RR$.
\end{theorem}

\begin{proof}
By Corollary~\ref{cor:ActionOfH0OnSphericalVectors} it follows that the representation $I(\nu)$ is irreducible if
\begin{align*}
 \nu\notin-\NN_0-\frac{p}{2}+(j-1)\frac{d}{2} && \mbox{and} && \nu\notin\NN_0+\frac{p}{2}-(j-1)\frac{d}{2},
\end{align*}
because then every $K$-type $V^{\bf m}$ can be reached from every other $K$-type $V^{\bf n}$ by successive application of $\td\rho_\nu$. Now suppose that $\nu=-m-\frac{p}{2}+(j-1)\frac{d}{2}$ for some $m\in\NN_0$ and $1\leq j\leq r$ (the other case is handled similarly). Then by Proposition~\ref{prop:ExplicitFormularhosInTermsOfOmega} we have
\begin{align*}
 P_{{\bf m}+e_j}\circ\td\rho_\nu(T)|_{V^{\bf m}} &= 0, & T\in\frakg_\CC
\end{align*}
for all ${\bf m}\geq0$ with $m_j=m$. Hence, the proper subspace of $I(\nu)$ consisting of all $K$-type $V^{\bf m}$ with ${\bf m}\geq0$, $m_j\leq m$, is $\frakg$-stable and therefore $I(\nu)$ is reducible. Finally note that, since
\begin{align*}
 \frac{p}{2}-(j-1)\frac{d}{2} &= (r-j)\frac{d}{2}+\frac{e+1}{2}+\frac{b}{4} \geq \frac{b}{4} > 0 & \forall\,1\leq j\leq r,
\end{align*}
the two possibilities cannot occur simultaneously.
\end{proof}

\subsection{Composition series}

Using the observations from the proof of Theorem~\ref{thm:Reducibility} it is easy to determine the composition series in the case where $I(\nu)$ is reducible. For this we let
\begin{align*}
 \ell_j(\nu) &:= \Set{{\bf m}\geq0}{m_j\leq-\nu-\frac{p}{2}+(j-1)\frac{d}{2}},\\
 r_j(\nu) &:= \Set{{\bf m}\geq0}{m_j\geq \nu-\frac{p}{2}+(j-1)\frac{d}{2}+1}
\end{align*}
for $1\leq j\leq r$ and define the following subspaces
\begin{align*}
 L_j(\nu) &:= \bigoplus_{{\bf m}\in\ell_j(\nu)}{V^{\bf m}}, & R_j(\nu) &:= \bigoplus_{{\bf m}\in r_j(\nu)}{V^{\bf m}}.
\end{align*}
Then we have the inclusions
\begin{align}
 \{0\} \subseteq L_1(\nu) \subseteq \cdots \subseteq L_r(\nu) \subseteq I(\nu)_{K\textup{-finite}},\label{eq:CompositionSeriesL}\\
 I(\nu)_{K\textup{-finite}} \supseteq R_1(\nu) \supseteq \cdots \supseteq R_r(\nu) \supseteq \{0\}.\label{eq:CompositionSeriesR}
\end{align}
The following result is immediate

\begin{theorem}\label{thm:CompositionSeries}
Let $X$ be a non-unital symmetric $R$-space and assume $I(\nu)$ is reducible.
\begin{enumerate}[(1)]
\item If $\nu\in-\NN_0-\frac{p}{2}+(j-1)\frac{d}{2}$ let $1\leq j_1<\cdots<j_t\leq r$ be such that $\nu\in-\NN_0-\frac{p}{2}+(j-1)\frac{d}{2}$ if and only if $j\in\{j_1,\ldots,j_t\}$. Then the composition series of $I(\nu)_{K\textup{-finite}}$ is given by
\begin{align*}
 \{0\} \subseteq L_{j_1}(\nu) \subseteq \cdots \subseteq L_{j_t}(\nu) \subseteq I(\nu)_{K\textup{-finite}}.
\end{align*}
\item If $\nu\in\NN_0+\frac{p}{2}-(j-1)\frac{d}{2}$ let $1\leq j_1<\cdots<j_t\leq r$ be such that $\nu\in\NN_0+\frac{p}{2}-(j-1)\frac{d}{2}$ if and only if $j\in\{j_1,\ldots,j_t\}$. Then the composition series of $I(\nu)_{K\textup{-finite}}$ is given by
\begin{align*}
 I(\nu)_{K\textup{-finite}} \supseteq R_{j_1}(\nu) \supseteq \cdots \supseteq R_{j_t}(\nu) \supseteq \{0\}.
\end{align*}
\end{enumerate}
\end{theorem}

\section{Unitarity}\label{sec:Unitarity}

Let $W\subseteq I(\nu)$ be a subrepresentation. Then $W$ is unitary if and only if there exists an intertwining operator $T:W\to I(-\overline{\nu})$ with strictly positive eigenvalues on the $K$-types occurring in $W$. In fact, in this case the invariant inner product on $W$ is given by
\begin{align*}
 (v,w)_W &:= \langle v,Tw\rangle_{L^2(X)} = \sum_{W({\bf m})\neq0}t({\bf m})\langle v_{\bf m},w_{\bf m}\rangle_{L^2(X)}, & v,w\in W,
\end{align*}
where $W({\bf m})$ denotes the $K$-isotypic component of $V^{\bf m}$ in $W$, $v_{\bf m}\in W({\bf m})$ the orthogonal projection (with respect to the $L^2$-inner product) of $v\in W$ onto $W({\bf m})$, and $t({\bf m})$ the eigenvalue of $T$ on $W({\bf m})$. To find such intertwining operators and their eigenvalues we employ the method of the ``spectrum generating operator'' by Branson--\'{O}lafsson--\O rsted~\cite{BOO96}. For this recall the eigenvalues $\pi_{\bf m}$ of the spectrum generating operator $\calR_\calP$ from Section~\ref{sec:SpectrumGeneratingOperator}.

\begin{theorem}[{\cite[Theorem 2.7]{BOO96}}]\label{thm:EigenvalueRatioIntertwiner}
Let $W\subseteq I(\nu)$ be $\frakg$-invariant. Then a map $T:W\to I(-\overline{\nu})$ with $T|_{V^{\bf m}}=t({\bf m})\id_{V^{\bf m}}$ is an intertwining operator if and only if for any ${\bf m},{\bf n}\geq0$ with $W({\bf m}),W({\bf n})\neq0$ and ${\bf n}\in S_{\bf m}$ we have
\begin{equation}
 (\pi_{\bf n}-\pi_{\bf m}+\tfrac{2n}{p}\nu)t({\bf n}) = (\pi_{\bf n}-\pi_{\bf m}-\tfrac{2n}{p}\overline{\nu})t({\bf m}).\label{eq:EigenvalueRatioIntertwiner}
\end{equation}
\end{theorem}

This immediately gives a necessary criterion on the coefficients $C({\bf m})$ in \eqref{eq:ActionOfH0OnSphericalVectors}:

\begin{lemma}\label{lem:CmConditionForUnitarity}
Let $W\subseteq I(\nu)$ be a unitary subrepresentation of $I(\nu)$. Then either $\nu\in i\RR$ or $C({\bf m})=0$ for all ${\bf m}\geq0$ with $W({\bf m})\neq0$.
\end{lemma}

\begin{proof}
Let $T:W\to I(-\overline{\nu})$ be an intertwining operator with strictly positive eigenvalues $t({\bf m})$ on $W({\bf m})\neq0$. Suppose there exists ${\bf m}\geq0$ with $W({\bf m})\neq0$ and $C({\bf m})\neq0$. Then by Corollary~\ref{cor:Sm} we find that ${\bf n}:={\bf m}\in S_{\bf m}$ and in this case \eqref{eq:EigenvalueRatioIntertwiner} implies
\begin{equation*}
 (\nu+\overline{\nu})t({\bf m}) = 0
\end{equation*}
which is only possible for $\nu\in i\RR$ since $t({\bf m})>0$.
\end{proof}

Note that if $C({\bf m})=0$ then by Corollary~\ref{cor:Sm} we have ${\bf m}\notin S_{\bf m}$ and the obstruction in the proof of Lemma~\ref{lem:CmConditionForUnitarity} does not occur. In this case \eqref{eq:EigenvalueRatioIntertwiner} gives recurrence relations for the eigenvalues $t({\bf m})$ which determine $T$ uniquely up to scalar multiples on every irreducible subrepresentation.

Thus we need to find out for which ${\bf m}$ the constants $C({\bf m})$ vanish. If $X$ is unital then $P$ is conjugate to its opposite parabolic and hence there exist Knapp--Stein intertwiners $I(\nu)\to I(-\nu)$ for infinitely many values $\nu\in\RR$. Therefore $C({\bf m})=0$ in these cases. For non-unital $X$, however, it can very well happen that $C({\bf m})\neq0$. From the formula for $C({\bf m})$ given in Lemma~\ref{lem:Vretare} it is hard to determine for which ${\bf m}$ we have $C({\bf m})=0$. We use the following result which is proved in Appendix~\ref{sec:CmFormula}:

\begin{proposition}\label{prop:CmSuperFormula}
Assume $2m_r+2\rho_r>1$.
\begin{enumerate}[(1)]
\item For $d\neq0,2$ we have
\begin{equation*}
 C({\bf m}) = \frac{rb}{2n}-\frac{pb(e+\frac{b}{2}-1)}{2nd(d-2)}\left(1-\prod_{k=1}^r\frac{(2m_k+2\rho_k)^2-(d-1)^2}{(2m_k+2\rho_k)^2-1}\right).
\end{equation*}
\item For $d\in\{0,2\}$ we have
\begin{equation*}
 C({\bf m}) = \frac{rb}{2n}+\frac{pb(e+\frac{b}{2}-1)}{2n}\sum_{k=1}^r\frac{1}{1-(2m_k+2\rho_k)^2}.
\end{equation*}
\end{enumerate}
\end{proposition}

We now study the constants $C({\bf m})$ separately for all non-unital symmetric $R$-spaces and determine all unitarizable subrepresentations using Theorem~\ref{thm:EigenvalueRatioIntertwiner}.

\subsection{The cases $\frakg=\sl(r+s,\KK)$ and $\frakg=\frake_{6(-26)}$}\label{sec:UnitarityGrassmannians}

Let $\frakg=\sl(r+s,\KK)$ with parabolic corresponding to $\frakl=\fraks(\gl(r,\KK)\oplus\gl(s,\KK))$ for $\KK\in\{\RR,\CC,\HH\}$ and $s>r\geq1$ or $\KK=\OO$ and $r=1$, $s=2$. Here $\sl(3,\OO)\cong\frake_{6(-26)}$.

\begin{lemma}\label{lem:CmForSLpq}
For $\frakg=\sl(r+s,\KK)$ with parabolic corresponding to $\frakl=\fraks(\gl(r,\KK)\oplus\gl(s,\KK))$ for $\KK\in\{\RR,\CC,\HH\}$ and $s>r\geq1$ or $\KK=\OO$ and $r=1$, $s=2$ we have $C({\bf m})=0$ if and only if ${\bf m}=0$.
\end{lemma}

\begin{proof}
For $\frakg=\sl(r+s,\RR)$ and $\frakg=\sl(r+s,\CC)$ this follows from \cite[Propositions 2.5.7 \&\ 3.5.3]{HL99} and for $\frakg=\sl(r+s,\HH)$ from \cite[Proposition 3.14.2]{Lee07}. It remains to consider the rank $1$ case $\frakg=\frake_{6(-26)}$. Here $n=16$, $p=12$, $r=1$, $d=0$, $e=7$ and $b=8$ and with Lemma~\ref{lem:Vretare} or Proposition~\ref{prop:CmSuperFormula} (2) we find after a short calculation
\begin{align*}
 C(m) &= \frac{m(m+11)}{4(m+5)(m+6)}
\end{align*}
which shows the claim for this last case.
\end{proof}

\begin{theorem}\label{thm:UnitarySubrepsSL}
Let $\frakg=\sl(r+s,\KK)$ with parabolic corresponding to $\frakl=\fraks(\gl(r,\KK)\oplus\gl(s,\KK))$ for $\KK\in\{\RR,\CC,\HH\}$ and $s>r\geq1$ or $\KK=\OO$ and $r=1$, $s=2$. Then $I(\nu)$ is irreducible and unitary if and only if $\nu\in i\RR$. If $I(\nu)$ is reducible, the only unitary subrepresentation resp. subquotient that can occur is the trivial representation and it occurs as subrepresentation resp. subquotient for $\nu=-\frac{p}{2}$ resp. $\nu=\frac{p}{2}$.
\end{theorem}

\begin{proof}
With Lemma~\ref{lem:CmConditionForUnitarity} and Lemma~\ref{lem:CmForSLpq} we find that either $\nu\in i\RR$ or the subrepresentation resp. subquotient is isomorphic to $V^{\bf 0}$, the trivial representation. That the trivial representation occurs as subrepresentation for $\nu=-\frac{p}{2}$ follows from Theorem~\ref{thm:CompositionSeries}.
\end{proof}

\begin{remark}
For $\frakg=\sl(r+s,\RR)$ and $\frakg=\sl(r+s,\CC)$ Theorem~\ref{thm:UnitarySubrepsSL} was shown by Howe--Lee~\cite[Propositions 2.5.7 \&\ 3.5.3]{HL99} and for $\frakg=\sl(r+s,\HH)$ by Lee~\cite[Proposition 3.14.2]{Lee07} (see also \cite{DZ97,Pas99,vDM99} for the case $r=1$). The case $\frakg=\frake_{6(-26)}$ does not seem to have been treated before.
\end{remark}

\subsection{The cases $\frakg=\frake_{6(6)}$ and $\frakg=\frake_6(\CC)$}\label{sec:UnitarityExceptional}

Next we study the two exceptional cases $\frakg=\frake_{6(6)}$ and $\frakg=\frake_6(\CC)$ of split rank $r=2$.

\begin{lemma}\label{lem:CmForRank2exceptional}
Suppose $\frakg=\frake_{6(6)}$ or $\frakg=\frake_6(\CC)$. Then $C({\bf m})=0$ if and only if $m_2=0$.
\end{lemma}

\begin{proof}
We use Proposition~\ref{prop:CmSuperFormula}~(1) to calculate $C({\bf m})$ explicitly.
\begin{enumerate}[(1)]
\item For $\frakg=\frake_{6(6)}$ we have $n=16$, $p=6$, $r=2$, $e=0$, $b=4$ and $d=3$, and hence
\begin{align*}
 C({\bf m}) &= \frac{m_2(2m_1+3)(2m_1+7)(m_2+2)}{4(m_1+2)(m_1+3)(2m_2+1)(2m_2+3)}
\end{align*}
from which the claim directly follows.
\item For $\frakg=\frake_6(\CC)$ we have $n=32$, $p=12$, $r=2$, $e=1$, $b=8$ and $d=6$, and obtain
\begin{align*}
 C({\bf m}) &= \frac{m_2(m_1+3)(m_1+8)(m_2+5)}{4(m_1+5)(m_1+6)(m_2+2)(m_2+3)}.
\end{align*}
This completes the proof.\qedhere
\end{enumerate}
\end{proof}

Recall the definition of the subspaces $L_j(\nu)$ from Section~\ref{sec:CompositionSeries}.

\begin{theorem}\label{thm:UnitarySubrepsE6}
Let either $\frakg=\frake_{6(6)}$ or $\frakg=\frake_6(\CC)$. Then $I(\nu)$ is irreducible and unitary if and only if $\nu\in i\RR$. The only unitary subrepresentations that occur in a reducible $I(\nu)$ are given as follows:
\begin{enumerate}[(1)]
\item Let $\frakg=\frake_{6(6)}$.
\begin{itemize}
\item The subrepresentation $L_1(-3)$ is the trivial representation and hence unitary.
\item The subrepresentation $L_2(-\frac{3}{2})$ with $K$-type decomposition and norm given by
\begin{equation*}
 L_2(-\tfrac{3}{2}) = \bigoplus_{m=0}^\infty{V^{(m,0)}}, \qquad \left\|\sum_{m=0}^\infty{v_m}\right\|^2 = \sum_{m=0}^\infty{\frac{\Gamma(m+\frac{9}{2})}{\Gamma(m+\frac{3}{2})}\|v_m\|_{L^2(X)}^2}
\end{equation*}
for $v_m\in V^{(m,0)}$ is unitary.
\end{itemize}
\item Let $\frakg=\frake_6(\CC)$.
\begin{itemize}
\item The subrepresentation $L_1(-6)$ is the trivial representation and hence unitary.
\item The subrepresentation $L_2(-3)$ with $K$-type decomposition and norm given by
\begin{equation*}
 L_2(-3) = \bigoplus_{m=0}^\infty{V^{(m,0)}}, \qquad \left\|\sum_{m=0}^\infty{v_m}\right\|^2 = \sum_{m=0}^\infty{\frac{\Gamma(m+9)}{\Gamma(m+3)}\|v_m\|_{L^2(X)}^2}
\end{equation*}
for $v_m\in V^{(m,0)}$ is unitary.
\end{itemize}
\end{enumerate}
\end{theorem}

\begin{proof}
By Lemma~\ref{lem:CmConditionForUnitarity} and Lemma~\ref{lem:CmForRank2exceptional} we find that a subrepresentation $W\subseteq I(\nu)$ can only be unitary if either $\nu\in i\RR$ or $W$ only contains $K$-types of the form $V^{(m,0)}$, $m\in\NN_0$.
\begin{enumerate}[(1)]
\item For $\frakg=\frake_{6(6)}$ this is by Theorem~\ref{thm:CompositionSeries} precisely the case if either $\nu=-3$ or $\nu=-\frac{3}{2}$. For $\nu=-3$ the subrepresentation $L_1(-\frac{3}{2})$ is the trivial representation and hence unitarizable. For $\nu=-\frac{3}{2}$ the subrepresentation $L_2(-\frac{3}{2})$ contains the $K$-types $V^{(m,0)}$ for $m\in\NN_0$ and it remains to show that this subrepresentation is unitary. By Theorem~\ref{thm:EigenvalueRatioIntertwiner} we obtain the following recurrence relation for the eigenvalues $t(m)$ of an intertwiner $T:L_2(-\frac{3}{2})\to I(\frac{3}{2})$ on $V^{(m,0)}$:
\begin{equation*}
 \frac{t(m+1)}{t(m)} = \frac{\pi_{(m+1,0)}-\pi_{(m,0)}-\frac{2n}{p}\nu}{\pi_{(m+1,0)}-\pi_{(m,0)}+\frac{2n}{p}\nu} = \frac{2m+p+3}{2m+p-3} = \frac{m+\frac{9}{2}}{m+\frac{3}{2}}.
\end{equation*}
Choosing $t(m):=\Gamma(m+\frac{9}{2})\Gamma(m+\frac{3}{2})^{-1}$ we obtain an intertwiner with strictly positive eigenvalues and hence an invariant inner product on $L_2(-\frac{3}{2})$.
\item Using the same argument for $\frakg=\frake_6(\CC)$ we find for $\nu=-3$ the subrepresentation $L_2(-3)$ with $K$-types $V^{(m,0)}$ and every operator $T:L_2(-3)\to I(3)$ with eigenvalue $t(m)$ on $V^{(m,0)}$ satisfying
\begin{align*}
 \frac{t(m+1)}{t(m)} = \frac{\pi_{(m+1,0)}-\pi_{(m,0)}-\frac{2n}{p}\nu}{\pi_{(m+1,0)}-\pi_{(m,0)}+\frac{2n}{p}\nu} = \frac{2m+p+6}{2m+p-6} = \frac{m+9}{m+3}
\end{align*}
provides an invariant form on $L_2(-3)$.\qedhere
\end{enumerate}
\end{proof}

\subsection{The cases $\frakg=\so(2r+1,2r+1)$ and $\frakg=\so(4r+2,\CC)$}\label{sec:UnitarityOrthogonal}

Finally we treat the remaining cases $\frakg=\so(2r+1,2r+1)$ and $\frakg=\so(4r+2,\CC)$.

\begin{lemma}\label{lem:CmForSO}
Suppose that $\frakg=\so(2r+1,2r+1)$ or $\frakg=\so(4r+2,\CC)$. Then $C({\bf m})=0$ if and only if $m_r=0$.
\end{lemma}

\begin{proof}
Again we calculate $C({\bf m})$ explicitly using Proposition~\ref{prop:CmSuperFormula}.
\begin{enumerate}[(1)]
\item For $\frakg=\so(2r+1,2r+1)$ the rank of $X$ is $r$ and $n=r(2r+1)$, $p=2r$, $e=0$, $b=d=2$, $\rho_k=\frac{2r-2k+1}{2}$. Since in this case $e+\frac{b}{2}-1=0$ it follows from Proposition~\ref{prop:CmSuperFormula}~(2) that for $m_r>0$
\begin{equation}
 C({\bf m}) = \frac{rb}{2n} = \frac{1}{2r+1}.\label{eq:CmForSOWithPositiveMr}
\end{equation}
Now let $m_r=0$. Since $2m_r+2\rho_r=1$ we cannot directly apply Proposition~\ref{prop:CmSuperFormula}~(2). Using the formula in Lemma~\ref{lem:Vretare} we find that
\begin{equation*}
 B({\bf m},r) = \frac{p}{2n}\frac{2m_r-1}{2m_r+1}\prod_{j<r}\frac{(2m_r-1)^2-(2m_j+2\rho_j)^2}{(2m_r+1)^2-(2m_j+2\rho_j)^2} = -\frac{1}{2r+1}
\end{equation*}
and hence non-zero. In the derivation of the formula in Proposition~\ref{prop:CmSuperFormula}~(2) the term $B({\bf m},r)$ was taken into account although ${\bf m}-e_r\ngeq0$. Therefore we can determine $C({\bf m})$ by adding $B({\bf m},r)$ to \eqref{eq:CmForSOWithPositiveMr} and obtain
\begin{align*}
 C({\bf m}) = \frac{1}{2r+1}+B({\bf m},r) = 0.
\end{align*}
\item For $\frakg=\so(4r+2,\CC)$ the rank of $X$ is $r$ and $n=2r(2r+1)$, $p=4r$, $e=1$, $b=d=4$, $\rho_k=\frac{4r-4k+3}{2}$ and we find with Proposition~\ref{prop:CmSuperFormula}~(1)
\begin{align*}
 C({\bf m}) &= \frac{1}{2r+1}\prod_{j=1}^r\frac{(m_j+2r-2j)(m_j+2r-2j+3)}{(m_j+2r-2j+1)(m_j+2r-2j+2)}.
\end{align*}
For ${\bf m}\geq0$ all factors in the product are strictly positive except the one for $j=r$ which is equal to
\begin{equation*}
 \frac{m_r(m_r+3)}{(m_r+1)(m_r+2)}.
\end{equation*}
Hence the whole product vanishes if and only if $m_r=0$.\qedhere
\end{enumerate}
\end{proof}

\begin{remark}
For $\frakg=\so(2r+1,2r+1)$ Lemma~\ref{lem:CmForSO} was already shown by Johnson~\cite[Proposition~3.6~(iii)]{Joh90} using different methods.
\end{remark}

To express the invariant inner product in these cases we define the Siegel--Gindikin Gamma function $\Gamma_{k,d}({\bf s})$ as a meromorphic function of ${\bf s}\in\CC^k$ by
\begin{equation*}
 \Gamma_{k,d}({\bf s}) := \prod_{j=1}^k\Gamma\left(s_j-(j-1)\frac{d}{2}\right).
\end{equation*}
We further identify a complex number $\sigma\in\CC$ with the vector $(\sigma,\ldots,\sigma)\in\CC^k$.

\begin{theorem}\label{thm:UnitarySubrepsSO}
Let either $\frakg=\so(2r+1,2r+1)$ or $\frakg=\so(4r+2,\CC)$. Then $I(\nu)$ is irreducible and unitary if and only if $\nu\in i\RR$. The only unitary subrepresentations that occur in a reducible $I(\nu)$ are given as follows:
\begin{enumerate}[(1)]
\item Let $\frakg=\so(2r+1,2r+1)$. For $j=0,1,\ldots,r-1$ the subrepresentation $L_{j+1}(-(r-j))$ with $K$-type decomposition and norm given by
\begin{align*}
 L_{j+1}(-(r-j)) &= \bigoplus_{{\bf m}\geq0, m_{j+1}=0}{V^{\bf m}},\\
 \left\|\sum_{{\bf m}\geq0,\,m_{j+1}=0}{v_{\bf m}}\right\|^2 &= \sum_{{\bf m}\geq0,\,m_{j+1}=0}^\infty{\frac{\Gamma_{j+1,2}({\bf m}+(2r-j))}{\Gamma_{j+1,2}({\bf m}+j)}\|v_{\bf m}\|_{L^2(X)}^2}
\end{align*}
for $v_{\bf m}\in V^{\bf m}$ is unitary.
\item Let $\frakg=\so(4r+2,\CC)$. For $j=0,1,\ldots,r-1$ the subrepresentation $L_{j+1}(-2(r-j))$ with $K$-type decomposition and norm given by
\begin{align*}
 L_{j+1}(-2(r-j)) &= \bigoplus_{{\bf m}\geq0, m_{j+1}=0}{V^{\bf m}},\\
 \left\|\sum_{{\bf m}\geq0,\,m_{j+1}=0}{v_{\bf m}}\right\|^2 &= \sum_{{\bf m}\geq0,\,m_{j+1}=0}^\infty{\frac{\Gamma_{j+1,4}({\bf m}+2(2r-j))}{\Gamma_{j+1,4}({\bf m}+2j)}\|v_{\bf m}\|_{L^2(X)}^2}
\end{align*}
for $v_{\bf m}\in V^{\bf m}$ is unitary.
\end{enumerate}
\end{theorem}

\begin{proof}
Using again Lemma~\ref{lem:CmConditionForUnitarity} and Lemma~\ref{lem:CmForSO} we find that a subrepresentation $W\subseteq I(\nu)$ can only be unitary if either $\nu\in i\RR$ or $W$ only contains $K$-types of the form $V^{\bf m}$ with $m_r=0$.
\begin{enumerate}[(1)]
\item Using Theorem~\ref{thm:CompositionSeries} we find that for $\frakg=\so(2r+1,2r+1)$ the subrepresentations of $I(\nu)$ containing only $V^{\bf m}$ with $m_r=0$ occur for $\nu=-(r-j)$, $j=0,1,\ldots,r-1$, and are given by $L_{j+1}(-(r-j))$. As in the proof of Theorem~\ref{thm:UnitarySubrepsE6} we use Theorem~\ref{thm:EigenvalueRatioIntertwiner} to find that a map $T:L_{j+1}(-(r-j))\to I((r-j))$ which acts by a scalar $t({\bf m})$ on the $K$-type $V^{\bf m}$ is an intertwining operator if and only if
\begin{align*}
 \frac{t({\bf m}+e_k)}{t({\bf m})} &= \frac{\pi_{{\bf m}+e_k}-\pi_{\bf m}-\frac{2n}{p}\nu}{\pi_{{\bf m}+e_k}-\pi_{\bf m}+\frac{2n}{p}\nu} = \frac{m_k+(2r-k-j+1)}{m_k+(j-k+1)}.
\end{align*}
Choosing
\begin{equation*}
 t({\bf m}) := \frac{\Gamma_{j+1,2}({\bf m}+(2r-j))}{\Gamma_{j+1,2}({\bf m}+j)}
\end{equation*}
we obtain an intertwiner $T:L_{j+1}(-(r-j))\to I((r-j))$ with strictly positive eigenvalues and therefore it provides an invariant inner product on $L_{j+1}(-(r-j))$.
\item Using Theorem~\ref{thm:CompositionSeries} we find that for $\frakg=\so(4r+2,\CC)$ the subrepresentations of $I(\nu)$ containing only $V^{\bf m}$ with $m_r=0$ occur for $\nu=-2(r-j)$, $j=0,1,\ldots,r-1$, and are given by $L_{j+1}(-2(r-j))$. The condition~\eqref{eq:EigenvalueRatioIntertwiner} for the eigenvalues $t({\bf m})$ of an intertwiner $T:L_{j+1}(-2(r-j))$ on the $K$-type $V^{\bf m}$ reads
\begin{align*}
 \frac{t({\bf m}+e_k)}{t({\bf m})} &= \frac{\pi_{{\bf m}+e_k}-\pi_{\bf m}-\frac{2n}{p}\nu}{\pi_{{\bf m}+e_k}-\pi_{\bf m}+\frac{2n}{p}\nu} = \frac{m_k+2(2r-k-j+1)}{m_k+2(j-k+1)}
\end{align*}
which shows the claim by the same method as in (1).\qedhere
\end{enumerate}
\end{proof}

\section{Gelfand--Kirillov dimension and associated varieties}\label{sec:AssocVar}

We calculate the Gelfand--Kirillov dimension of some unitary representations we found in the previous section and use it to find their associated variety.

\subsection{Gelfand--Kirillov dimension}

Let us briefly recall the definition of the Gelfand--Kirillov dimension (see \cite{Vog78} for details). We denote by
\begin{align*}
 0=\calU_{-1}(\frakg)\subseteq\calU_0(\frakg)\subseteq\calU_1(\frakg)\subseteq\cdots\subseteq\calU_k(\frakg)\subseteq\cdots
\end{align*}
the canonical filtration of the universal enveloping algebra $\calU(\frakg)$ of $\frakg$. For a finitely generated $\calU(\frakg)$-module $W$ choose a finite-dimensional generating subspace $W_0\subseteq W$ and define the Gelfand--Kirillov dimension of $W$ by
\begin{align*}
 \GKdim(W) &:= \limsup_{k\to\infty}{\frac{\log\dim\calU_k(\frakg)W_0}{\log k}}.
\end{align*}
This definition does not depend on the choice of $W_0$.

For an irreducible unitary representation $(\pi,\calH)$ of $G$ let $\calH_{K\textup{-finite}}$ denote the space of $K$-finite vectors. $\calH_{K\textup{-finite}}$ is a finitely generated $\calU(\frakg)$-module and we define the Gelfand--Kirillov dimension of $\pi$ by
\begin{equation*}
 \GKdim(\pi) := \GKdim(\calH_{K\textup{-finite}}).
\end{equation*}

\subsection{Associated variety and nilpotent orbits}

The Gelfand--Kirillov dimension of an irreducible unitary representation $\pi$ is closely related to the associated variety of $\pi$. Let us recall the definition of the associated variety (see \cite{Vog91} for details). For a finitely generated $\calU(\frakg)$-module $W$ we again choose a finite-dimensional generating subspace $W_0\subseteq W$ and define a filtration $W=\cup_{k=0}^\infty{W_k}$ of $W$ by $W_k:=\calU_k(\frakg)W_0$. The graded algebra $\gr\calU(\frakg)=\bigoplus_{k=0}^\infty{\calU_k(\frakg)/\calU_{k-1}(\frakg)}$ is isomorphic to $S(\frakg_\CC)$ and therefore we can regard the corresponding graded module $\gr W=\bigoplus_{k=0}^\infty{W_k/W_{k-1}}$ as a module over $S(\frakg_\CC)$. Denote by $J:=\Ann_{S(\frakg_\CC)}(\gr W)$ the annihilator ideal of $\gr W$ in $S(\frakg_\CC)$. The affine variety $\calV(J)\subseteq\frakg_\CC^*$ corresponding to the ideal $J\subseteq S(\frakg_\CC)\cong\CC[\frakg_\CC^*]$ is called the \textit{associated variety} of $W$. It does not depend on the choice of the generating subspace $W_0$ (see \cite[Proposition 2.2]{Vog91}).

We define the associated variety $\calV(\pi)$ of an irreducible unitary representation $(\pi,\calH)$ to be the associated variety of the $\calU(\frakg)$-module $\calH_{K\textup{-finite}}$. The variety $\calV(\pi)$ is always a $K_\CC$-stable closed subvariety of $\frakp_\CC^*$ consisting of nilpotent elements and hence the union of finitely many nilpotent $K_\CC$-orbits (see \cite[Corollary 5.23]{Vog91}). All irreducible components of the associated variety $\calV(\pi)$ have the same dimension and the relation to the Gelfand--Kirillov dimension of $\pi$ is given by
\begin{equation*}
 \GKdim(\pi) = \dim\calV(\pi).
\end{equation*}

Assume that $G$ is non-Hermitian. Then there exists a unique non-zero nilpotent $K_\CC$-orbit $\calO_{\min}^{K_\CC}\subseteq\frakp_\CC^*$ of minimal dimension $m(\frakg)$ (see \cite[Proposition 2.2]{KO12}). If $\frakg_\CC$ is a simple complex Lie algebra then there is also a unique minimal nilpotent coadjoint orbit $\calO_{\min}^{G_\CC}\subseteq\frakg_\CC^*$. By results of Kostant--Rallis~\cite{KR71}, Brylinski~\cite{Bry98} and Okuda~\cite{Oku11} the two orbits are related as follows:

\begin{proposition}[{see \cite{Bry98,KR71,Oku11}}]\label{prop:MinOrbitWithRealPoints}
Let $\frakg$ be a non-Hermitian simple real Lie algebra. Then exactly one of the following three cases occurs:
\begin{enumerate}[(1)]
\item $\frakg_\CC$ is a simple complex Lie algebra and $\calO_{\min}^{G_\CC}\cap\frakp_\CC^*=\calO_{\min}^{K_\CC}$. In this case the dimension of $\calO_{\min}^{K_\CC}$ is half the dimension of $\calO_{\min}^{G_\CC}$ and hence given by the following table:
\begin{equation*}
\begin{array}{c|c|c|c|c|c|c|c|c|c}
\frakg_\CC & A_k & B_k\ (k\geq2) & C_k & D_k & \frakg_2^\CC & \frakf_4^\CC & \frake_6^\CC & \frake_7^\CC & \frake_8^\CC\\
\hline
\frac{1}{2}\dim\calO_{\min}^{G_\CC} & k & 2k-2 & k & 2k-3 & 3 & 8 & 11 & 17 & 29
\end{array}
\end{equation*}
\item $\frakg_\CC$ is a simple complex Lie algebra and $\calO_{\min}^{G_\CC}\cap\frakp_\CC^*=\emptyset$. This happens exactly for $\frakg$ in the following list
\begin{equation*}
\begin{array}{c|c|c|c|c|c}
\frakg & \su^*(2k) & \so(k+1,1) & \sp(p,q) & \frake_{6(-26)} & \frakf_{4(-20)}\\
\hline
\dim\calO_{\min}^{K_\CC} & 4k-4 & k & 2(p+q)-1 & 16 & 11
\end{array}
\end{equation*}
\item $\frakg=\frakh_\CC$ is the complexification of a simple real Lie algebra $\frakh$. In this case
\begin{equation*}
 \dim\calO_{\min}^{K_\CC} = 2\,m(\frakh).
\end{equation*}
\end{enumerate}
\end{proposition}

In fact, as shown by Okuda~\cite{Oku11}, in case (2) the orbit $\calO_{\min}^{K_\CC}$ is the intersection of the nilpotent coadjoint orbit $G_\CC\cdot\calO_{\min}^{K_\CC}\subseteq\frakg_\CC^*$ with $\frakp_\CC^*$. This orbit is the unique coadjoint orbit of minimal dimension in $\frakg_\CC^*$ which has nontrivial intersection with $\frakp_\CC^*$ (or equivalently with $\frakg^*$). For more details on nilpotent orbits we refer to the book by Collingwood-McGovern~\cite{CM93}.

A direct consequence of this is that for every infinite-dimensional irreducible unitary representation $(\pi,\calH)$ of $G$ we have
\begin{equation}
 \GKdim(\pi) \geq m(\frakg).\label{eq:LowerBoundGKDIM}
\end{equation}

\subsection{Small representations}

Using these observations we can now determine the associated variety of some unitary representations constructed in Section~\ref{sec:Unitarity}. For this we first calculate their Gelfand--Kirillov dimension.

\begin{proposition}\label{prop:EstimateGKDIM}
\begin{enumerate}[(1)]
\item Let $X$ be an arbitrary symmetric $R$-space. For $\nu\in i\RR$ we have
\begin{equation*}
 \GKdim(\pi_\nu) = n = \dim(\frakn).
\end{equation*}
\item For $\frakg=\frake_{6(6)}$, $\frakg=\frake(\CC)$, $\frakg=\so(2r+1,2r+1)$ and $\frakg=\so(4r+2,\CC)$ let $W$ be the unique unitary subrepresentation of $I(\nu)$ with $K$-types $W=\bigoplus_{m=0}^\infty{V^{(m,0,\ldots,0)}}$. Then
\begin{equation*}
 \GKdim(W) = m(\frakg).
\end{equation*}
\end{enumerate}
\end{proposition}

\begin{proof}
\begin{enumerate}[(1)]
\item This follows directly from the more general statement \cite[Lemma 6.5]{Vog78}.
\item The strategy for the other cases can be described as follows: First calculate the dimension of the $K$-types $V^{(m,0,\ldots,0)}$ as a function of $m$ and determine the highest power $D$ of $m$, i.e. $\dim V^{(m,0,\ldots,0)}=cm^D+o(m^D)$ with $c\neq0$. This can be done by using the Weyl dimension formula
\begin{equation}
 \dim V^{(m,0,\ldots,0)} = \prod_{\alpha\in\Delta^+(\frakk_\CC,\frakh_\CC)}{\frac{\langle\lambda_{(m,0,\ldots,0)}+\rho^\frakh,\alpha\rangle}{\langle\rho^\frakh,\alpha\rangle}}.\label{eq:WeylDimFormula}
\end{equation}
Here $\frakh_\CC\subseteq\frakk_\CC$ is a split Cartan subalgebra such that $\frakh_\CC=(\frakh_\CC\cap\frakm_\CC)\oplus\frakt_\CC$ with $\frakt$ the maximal torus in $\frakn_\frakk$ from \eqref{eq:DefTorusT}, where $\Delta^+(\frakk_\CC,\frakh_\CC)$ a set of positive roots for $\Delta(\frakk_\CC,\frakh_\CC)$, compatible with the ordering of $\frakt_\CC^*$, and $\rho^\frakh$ denotes half the sum of all positive roots. Next observe that if $W_0=V^{\bf 0}$ generates a subrepresentation $W\subseteq I(\nu)$ with $K$-types $V^{(m,0,\ldots,0)}$, $m\in\NN_0$, we have
\begin{equation*}
 \dim\calU_k(\frakg)W_0 \leq \dim\bigoplus_{m=0}^kV^{(m,0,\ldots,0)} = c'k^{D+1}+o(k^{D+1})
\end{equation*}
with $c'\neq0$ and hence
\begin{equation*}
 \GKdim(W) = \lim_{k\to\infty}{\frac{\log\dim\calU_k(\frakg)W_0}{\log k}} \leq \lim_{k\to\infty}\frac{\log k^{D+1}}{\log k} = D+1.
\end{equation*}
This gives an upper bound for the Gelfand--Kirillov dimension of $W$. Compare this upper bound with the lower bound~\eqref{eq:LowerBoundGKDIM} to find that $\GKdim(W)=D+1=m(\frakg)$.
\begin{enumerate}[(a)]
\item $\frakg=\frake_{6(6)}$. The Lie algebra $\frakk=\sp(4)$ is of type $C_4$. Its corresponding symmetric subalgebra is given by $\frakm=\sp(2)+\sp(2)$. Write $\frakk=\frakm\oplus\frakn_{\frakk}$ with $\frakn_\frakk$ as in \eqref{eq:nIntersection} and choose a split Cartan subalgebra $\frakh_\CC\subseteq\frakk_\CC$. We can label the roots $\Delta(\frakk_\CC,\frakh_\CC)=\set{\frac{\pm\varepsilon_j\pm\varepsilon_k}{2}}{1\leq j<k\leq4}\cup\set{\pm\varepsilon_j}{1\leq j\leq4}$ such that $\frakt_\CC^*=\CC(\varepsilon_1+\varepsilon_2)+\CC(\varepsilon_3+\varepsilon_4)$ and
\begin{align*}
 \varepsilon_1+\varepsilon_2+\varepsilon_3+\varepsilon_4 &= 2\gamma_1,\\
 \varepsilon_1+\varepsilon_2-\varepsilon_3-\varepsilon_4 &= 2\gamma_2,
\end{align*}
where $\gamma_1$ and $\gamma_2$ were defined in Section~\ref{sec:RootData}. As positive roots we may take $\Delta^+(\frakk_\CC,\frakh_\CC)=\set{\frac{\varepsilon_j\pm\varepsilon_k}{2}}{1\leq j<k\leq4}\cup\set{\varepsilon_j}{1\leq j\leq4}$. Then the highest weight $\lambda_{(m,0)}=m\gamma_1$ of $V^{(m,0)}$ is given by
\begin{equation*}
 \lambda_{(m,0)} = m\frac{\varepsilon_1+\varepsilon_2+\varepsilon_3+\varepsilon_4}{2}.
\end{equation*}
In the Weyl dimension formula~\eqref{eq:WeylDimFormula} the only terms contributing a power of $m$ to the dimension occur for $\alpha$ equal to
\begin{equation*}
 \varepsilon_j\ (j=1,2,3,4), \qquad \frac{\varepsilon_j+\varepsilon_k}{2}\ (1\leq j<k\leq 4).
\end{equation*}
Counting these roots yields $D=10$. Comparing with $m(\frakg)$ gives by Proposition~\ref{prop:MinOrbitWithRealPoints}~(1) that $D+1=11=m(\frakg)$.
\item $\frakg=\frake_6(\CC)$. Note that $(\frakk,\frakm)=(\frake_6,\so(10)+\RR)$ is a Hermitian symmetric pair. Therefore, the roots $\gamma_1$ and $\gamma_2$ defined in Section~\ref{sec:RootData} form a maximal system of strongly orthogonal non-compact roots. In the notation of Freudenthal--de Vries~\cite{FdV69} we have
\begin{equation*}
 \gamma_1 = \left(\begin{matrix}&&2&&\\1&2&3&2&1\end{matrix}\right), \qquad \gamma_2 = \left(\begin{matrix}&&0&&\\1&1&1&1&1\end{matrix}\right).
\end{equation*}
Let $\Delta^+(\frakk_\CC,\frakh_\CC)$ denote the positive roots of $\frake_6$ as in \cite[Table B]{FdV69}. Then $\lambda_{(m,0)}=m\gamma_1$ is the highest weight of $V^{(m,0)}$. Since all roots in $\Delta^+(\frakk_\CC,\frakh_\CC)$ have the same length and the angle between two simple roots is either $\frac{\pi}{2}$ or $\frac{2\pi}{3}$ a positive root $\alpha$ contributes a power of $m$ to the dimension in \eqref{eq:WeylDimFormula} if and only if
\begin{align*}
 \alpha &= \left(\begin{matrix}&&a&&\\ *&*&*&*&*\end{matrix}\right)
\end{align*}
with $a\neq0$. (In fact, only $a=1,2$ occur.) Counting these roots using \cite[Table B]{FdV69} yields $D=21$. Comparison with $m(\frakg)$ gives by Proposition~\ref{prop:MinOrbitWithRealPoints}~(3) that $D+1=22=m(\frakg)$.
\item $\frakg=\so(2r+1,2r+1)$. The Lie algebra $\frakk=\so(2r+1)+\so(2r+1)$ is of type $B_r\times B_r$. Its corresponding symmetric subalgebra is given by $\frakm=\so(2r+1)$. Write $\frakk=\frakm\oplus\frakn_{\frakk}$ with $\frakn_\frakk$ as in \eqref{eq:nIntersection} and choose a split Cartan subalgebra $\frakh_\CC\subseteq\frakk_\CC$. We can label the roots $\set{\pm\varepsilon_i}{1\leq i\leq2r}\cup\set{\pm\varepsilon_i\pm\varepsilon_j}{1\leq i<j\leq r\mbox{ or }r+1\leq i<j\leq2r}$ such that $\frakt_\CC^*=\sum_{i=1}^r{\CC(\varepsilon_i+\varepsilon_{i+r})}$ and
\begin{align*}
 \varepsilon_i+\varepsilon_{i+r} &= \gamma_i, & 1\leq i\leq r,
\end{align*}
where $\gamma_i$ were defined in Section~\ref{sec:RootData}. As positive roots we may choose $\Delta^+(\frakk_\CC,\frakh_\CC)=\set{\varepsilon_i}{1\leq i\leq2r}\cup\set{\varepsilon_i\pm\varepsilon_j}{1\leq i<j\leq r\mbox{ or }r+1\leq i<j\leq2r}$. The highest weight $\lambda_{(m,0,\ldots,0)}=m\gamma_1$ of $V^{(m,0,\ldots,0)}$ is given by
\begin{equation*}
 \lambda_{(m,0,\ldots,0)} = m(\varepsilon_1+\varepsilon_{r+1}).
\end{equation*}
In the Weyl dimension formula~\eqref{eq:WeylDimFormula} the only terms contributing a power of $m$ to the dimension occur for $\alpha$ equal to
\begin{equation*}
 \varepsilon_1,\quad\varepsilon_1\pm\varepsilon_i\ (2\leq i\leq r), \quad \varepsilon_{r+1},\quad\varepsilon_{r+1}\pm\varepsilon_i\ (r+2\leq i\leq2r).
\end{equation*}
Counting these roots yields $D=4r-2$ and by Proposition~\ref{prop:MinOrbitWithRealPoints}~(1) we find $D+1=4r-1=m(\frakg)$.
\item $\frakg=\so(4r+2,\CC)$. The Lie algebra $\frakk=\so(4r+2)$ is of type $D_{2r+1}$. Its corresponding symmetric subalgebra is given by $\frakm=\fraku(2r+1)$. Write $\frakk=\frakm\oplus\frakn_{\frakk}$ with $\frakn_\frakk$ as in \eqref{eq:nIntersection} and choose a split Cartan subalgebra $\frakh_\CC\subseteq\frakk_\CC$. We can label the roots $\set{\pm\varepsilon_i\pm\varepsilon_j}{1\leq i<j\leq2r+1}$ such that $\frakt_\CC^*=\sum_{i=1}^r{\CC(\varepsilon_{2i-1}+\varepsilon_{2i})}$ and
\begin{align*}
 \varepsilon_{2i-1}+\varepsilon_{2i} &= \gamma_i, & 1\leq i\leq r,
\end{align*}
where $\gamma_i$ were defined in Section~\ref{sec:RootData}. As positive roots we may choose $\Delta^+(\frakk_\CC,\frakh_\CC)=\set{\varepsilon_i\pm\varepsilon_j}{1\leq i<j\leq2r+1}$. Then the highest weight $\Lambda_{(m,0,\ldots,0)}=m\gamma_1$ of $V^{(m,0,\ldots,0)}$ is given by
\begin{equation*}
 \Lambda_{(m,0)} = m(\varepsilon_1+\varepsilon_2).
\end{equation*}
By the Weyl dimension formula~\eqref{eq:WeylDimFormula} the only terms contributing a power of $m$ to the dimension occur for $\alpha$ equal to
\begin{align*}
 \varepsilon_1+\varepsilon_2,\quad\varepsilon_1\pm\varepsilon_i\ (3\leq i\leq2r+1),\quad\varepsilon_2\pm\varepsilon_i\ (3\leq i\leq2r+1).
\end{align*}
Counting these roots yields $D=8r-3$. Comparing with $m(\frakg)$ gives by Proposition~\ref{prop:MinOrbitWithRealPoints}~(3) that $D+1=8r-2=m(\frakg)$.\qedhere
\end{enumerate}
\end{enumerate}
\end{proof}

\begin{theorem}\label{thm:AssociatedVariety}
Let $W$ be one of the following representations:
\begin{enumerate}[(1)]
\item $\frakg=\sl(1+s,\RR),\sl(1+s,\CC),\sl(1+s,\HH),\frake_{6(-26)}$: Let $W=I(\nu)$ with $\nu\in i\RR$, the unitary principal series associated to the corresponding non-unital symmetric $R$-space of rank $1$,
\item $\frakg=\frake_{6(6)},\frake_6(\CC),\so(2r+1,2r+1),\so(4r+2,\CC)$: Let $W$ be the unique unitary subrepresentation of $I(\nu)$ with $K$-types $W=\bigoplus_{m=0}^\infty{V^{(m,0,\ldots,0)}}$.
\end{enumerate}
Then the associated variety of $W$ is the closure of the minimal nilpotent $K_\CC$-orbit:
\begin{equation*}
 \calV(W) = \overline{\calO_{\min}^{K_\CC}}.
\end{equation*}
In particular, the Gelfand--Kirillov dimension of $W$ attains its minimum among all infinite-dimensional irreducible unitary representations of $G$ and any of its covering groups.
\end{theorem}

\begin{proof}
First note that $\calO_{\min}^{K_\CC}$ is the unique coadjoint orbit with smallest possible dimension that can occur in the associated variety of a non-trivial irreducible unitary representation. We show that the associated variety of the representation in question has in all cases the same dimension as $\calO_{\min}^{K_\CC}$ and hence the result follows. Since the dimension of the associated variety is equal to the Gelfand--Kirillov dimension of the representation it suffices to show that the Gelfand--Kirillov dimension coincides with the dimension $m(\frakg)$ of $\calO_{\min}^{K_\CC}$. For case (2) this is already explicitly stated in Proposition~\ref{prop:EstimateGKDIM}. For case (1) the statement in Proposition~\ref{prop:EstimateGKDIM}~(1) reduces this to show that $\dim\frakn=m(\frakg)$ which we check for the separate cases:
\begin{enumerate}[(a)]
\item $\frakg=\sl(1+s,\RR)$, $\frakn=M(1\times s,\RR)$. By Proposition~\ref{prop:MinOrbitWithRealPoints}~(1) the orbit $\calO_{\min}^{K_\CC}$ has dimension $s$ and hence $\dim\calO_{\min}^{K_\CC}=\dim\frakn$.
\item $\frakg=\sl(1+s,\CC)$, $\frakn=M(1\times s,\CC)$. We have $\dim\calO_{\min}^{K_\CC}=2s$ by Proposition~\ref{prop:MinOrbitWithRealPoints}~(3) and hence $\dim\calO_{\min}^{K_\CC}=\dim\frakn$.
\item $\frakg=\sl(1+s,\HH)$, $\frakn=M(1\times s,\HH)$. Since $\frakg\cong\su^*(2s+2)$ we have by Proposition~\ref{prop:MinOrbitWithRealPoints}~(2) that $\dim\calO_{\min}^{K_\CC}=4s=\dim\frakn$.
\item $\frakg=\frake_{6(-26)}$, $\frakn=M(1\times2,\OO)$. Again by Proposition~\ref{prop:MinOrbitWithRealPoints}~(2) we obtain $\dim\calO_{\min}^{K_\CC}=16=\dim\frakn$.\qedhere
\end{enumerate}
\end{proof}

\begin{remark}
For the groups $G=\SO(2r+1,2r+1)$ and $G=E_{6(6)}$ the irreducible unitary representation $W$ in Theorem~\ref{thm:AssociatedVariety} is in fact the minimal representation, i.e. the annihilator of $W$ in $\calU(\frakg)$ is equal to the Joseph ideal. This can e.g. be seen by comparing both infinitesimal character and associated variety of the annihilator of $W$ and the Joseph ideal. They both agree and hence, by a result of Duflo, this implies that the annihilator of $W$ is the Joseph ideal. However, we will give a different proof of minimality in a forthcoming paper which does not use the infinitesimal character. In both cases in question the complex Lie algebra $\frakg_\CC$ is simple and not of type $A$. Hence the Joseph ideal is the unique completely prime two-sided ideal in $\calU(\frakg)$ with associated variety equal to the minimal nilpotent coadjoint orbit in $\frakg_\CC^*$ by \cite[Proposition 10.2]{Jos76} (see also \cite[Theorem 3.1]{GS04}). From Theorem~\ref{thm:AssociatedVariety} it follows that the annihilator of $W$ in $\calU(\frakg)$ has associated variety equal to the minimal nilpotent coadjoint orbit in $\frakg_\CC^*$. Therefore it remains to show that the annihilator is a completely prime ideal in $\calU(\frakg)$. We will show that the representation $W$ can be realized by regular differential operators on an irreducible algebraic variety. The ring of those differential operators does not contain zero divisors and hence the annihilator of $W$ is completely prime. This method was applied before in \cite[Theorem 2.18]{HKM}.
\end{remark}

\appendix

\section{Proof of Proposition~\ref{prop:CmSuperFormula}}\label{sec:CmFormula}

This section is devoted to the proof of Proposition~\ref{prop:CmSuperFormula}. For this we put $x_k:=2(m_k+\rho_k)$, $\alpha:=\frac{b}{2}+1$, $\beta:=\frac{b}{2}+e$ and $\gamma:=d$. Then by Lemma~\ref{lem:Vretare}
\begin{equation*}
 C({\bf m}) = 1-\frac{p}{2n}d(x),
\end{equation*}
where we abbreviate
\begin{align*}
 d(x) &:= \sum_{k=1}^r{\left(a(x,k)+b(x,k)\right)},\\
 a(x,k) &:= \frac{(x_k + \alpha)(x_k+\beta)}{x_k\,(x_k+1)}\cdot\prod_{j\neq k}\frac{(x_k + \gamma)^2-x_j^2}{x_k^2-x_j^2},\\
 b(x,k) &:= \frac{(x_k - \alpha)(x_k-\beta)}{x_k\,(x_k-1)}\cdot\prod_{j\neq k}\frac{(x_k - \gamma)^2-x_j^2}{x_k^2-x_j^2}.
\end{align*}
Proposition~\ref{prop:CmSuperFormula} now follows from the following identities:

\begin{proposition}
Let $x\in\CC^r$ with $x_j\neq x_k$ for $j\neq k$ and $x_j\neq\pm1$. Further let $\alpha,\beta,\gamma\in\CC$.
\begin{enumerate}[(1)]
\item For $\gamma\neq0,2$ we have
	\[
		d(x) = 2r + \frac{2(1-\alpha)(1-\beta)}{\gamma(\gamma-2)}
					\left(1-\prod_{k=1}^r\frac{(\gamma-1)^2-x_k^2}{1-x_k^2}\right).
	\]
\item For $\gamma\in\{0,2\}$ we have
	\[
		d(x) = 2r - 2(1-\alpha)(1-\beta)\sum_{k=1}^r\frac{1}{1-x_k^2}.
	\]
\end{enumerate}
\end{proposition}
\begin{proof}
For convenience, we set $\underline r:=\{1,\ldots, r\}$. For fixed $k\in\underline r$ we first note that
\begin{align*}
	\prod_{j\neq k}\frac{(x_k \pm \gamma)^2-x_j^2}{x_k^2-x_j^2}
		= \prod_{j\neq k}\left(1 + \frac{\gamma(\gamma\pm 2x_k)}{x_k^2-x_j^2}\right)
		= \sum_{J\subseteq\underline r\setminus\{k\}}
				\frac{\gamma^{|J|}(\gamma\pm2x_k)^{|J|}}{\prod_{j\in J}(x_k^2-x_j^2)}
\end{align*}
Then a short calculation shows that 
\[
	\sum_{k=1}^r a(x,k) + b(x,k)
		= \sum_{k=1}^r \sum_{J\subseteq\underline r\setminus\{k\}}
			\frac{\gamma^{|J|}p_{|J|}(x_k)}{(x_k^2-1)\prod_{j\in J}(x_k^2-x_j^2)},
\]
where
\[
	p_m(x):=\frac{1}{x}\,\big((x+\alpha)(x+\beta)(x-1)(\gamma + 2x)^m
		+ (x-\alpha)(x-\beta)(x+1)(\gamma-2x)^m\big).
\]
We note that $p_m$ is an even polynomial of degree $\leq m+2$. Changing the order of summation by using the bijection $J\mapsto J\cup\{k\}$ between subsets of $\underline r$ not containing $k$ and those containing $k$, we obtain
\[
	d(x) = \sum_{\substack{J\subseteq\underline r,\\J\neq\emptyset}} \gamma^{|J|-1}\sum_{k\in J}
					\frac{p_{|J|-1}(x_k)}{(x_k^2-1)\prod_{j\in J\setminus\{k\}}(x_k^2-x_j^2)}.
\]
In order to evaluate the inner sum, we set $x_0:=1$ and $J_0:=J\cup\{0\}$, and find that
\[
	\sum_{k\in J}\frac{p_{|J|-1}(x_k)}{(x_k^2-1)\prod_{j\in J\setminus\{k\}}(x_k^2-x_j^2)}
		= \left(\sum_{k\in J_0}\frac{p_{|J|-1}(x_k)}{\prod_{j\in J_0\setminus\{k\}}(x_k^2-x_j^2)}\right)
			- \frac{p_{|J|-1}(x_0)}{\prod_{j\in J}(x_0^2-x_j^2)}.
\]
Due to the following lemma (applied to $\{y_1,\ldots,y_{|J_0|}\} = \{x_j\,|\,j\in J_0\}$), the first term vanishes if $|J|>1$. Therefore,
\[
	d(x) = \sum_{k=1}^r\left(\frac{p_0(x_k)}{x_k^2-x_0^2} + \frac{p_0(x_0)}{x_0^2-x_k^2}\right)
			- \sum_{\substack{J\subseteq\underline r\\J\neq\emptyset}}
				\frac{\gamma^{|J|-1}p_{|J|-1}(x_0)}{\prod_{j\in J}(x_0^2-x_j^2)}.
\]
Resubstituting $x_0$ by $1$, we obtain for the first term by an elementary calculation
\[
	\sum_{k=1}^r\left(\frac{p_0(x_k)}{x_k^2-x_0^2} + \frac{p_0(x_0)}{x_0^2-x_k^2}\right)
		= \sum_{k=1}^r 2 = 2r.
\]
Using $p_m(1) = 2(1-\alpha)(1-\beta)(\gamma-2)^m$ in the second term, we find
\begin{equation*}
	\sum_{\substack{J\subseteq\underline r\\J\neq\emptyset}}
			\frac{\gamma^{|J|-1}p_{|J|-1}(x_0)}{\prod_{j\in J}(1-x_j^2)}
	 = \frac{2(1-\alpha)(1-\beta)}{\prod_{k=1}^n (1-x_k^2)}
	 		\sum_{\substack{J\subseteq\underline r\\J\neq\emptyset}}
	 			(\gamma(\gamma-2))^{|J|-1}\prod_{j\in\underline r\setminus J}(1-x_j^2).
\end{equation*}
Now if $\gamma\in\{0,2\}$ only the summands for $|J|=1$ survive and the claimed formula follows. For $\gamma\neq0,2$ we find
\begin{multline*}
	\frac{2(1-\alpha)(1-\beta)}{\prod_{k=1}^n (1-x_k^2)}
	 		\sum_{\substack{J\subseteq\underline r\\J\neq\emptyset}}
	 			(\gamma(\gamma-2))^{|J|-1}\prod_{j\in\underline r\setminus J}(1-x_j^2)\\
	 = \frac{2(1-\alpha)(1-\beta)}{\gamma(\gamma-2)}\left(-1 + 
	 		\frac{\sum_{J\subseteq\underline r}
	 			(\gamma(\gamma-2))^{|J|}\prod_{j\in\underline r\setminus J}(1-x_j^2)}{\prod_{k=1}^n (1-x_k^2)}\right)
\end{multline*}
Finally, since
\[
	\sum_{J\subseteq\underline r}(\gamma(\gamma-2))^{|J|}
	\prod_{j\in\underline r\setminus J}(1-x_j^2)
	= \prod_{k=1}^r \big(\gamma(\gamma-2) + (1 - x_k^2)\big)
	= \prod_{k=1}^r ((\gamma-1)^2 - x_k^2),
\]
it follows that
\[
	d(x) = 2r + \frac{2(1-\alpha)(1-\beta)}{\gamma(\gamma-2)}
				\left(1-\prod_{k=1}^r\frac{(\gamma-1)^2 - x_k^2}{1-x_k^2}\right).
\]
This completes the proof.
\end{proof}

\begin{lemma}
	Let $y_1,\ldots, y_N$ be pairwise distinct real numbers. If $m<N-1$, then
	\[
		\sum_{k=1}^N\frac{y_k^m}{\prod_{j\neq k} (y_k-y_j)} = 0.
	\]
\end{lemma}
\begin{proof}
Let 
\[
	\ell_k(x):=\prod_{j\neq k}\frac{x-y_j}{y_k-y_j}
\]
be the $k$'th Lagrange polynomial. Then, $\ell_k(y_j) = \delta_{kj}$ and hence
\[
	\mbox{span}\{\ell_k\,|\,k=1,\ldots,N\}= \RR[X]_{\leq N-1},
\]
the space of polynomials of degree $\leq N-1$. For any $p\in\RR[X]_{\leq N-1}$ we obtain
\[
	p(x) = \sum_{k=1}^N p(y_k)\,\ell_k(x).
\]
Applied to $p(x):=-x^{m+1}$ with $m<N-1$ and evaluated at $x=0$, this yields
\[
	0 = \sum_{k=1}^N \frac{y_k^m}{\prod_{j\neq k} (y_k-y_j)}\cdot\prod_{j=1}^N (-y_j).
\]
We may assume that $y_j\neq 0$ for all $j$ and conclude the statement.
\end{proof}

\bibliographystyle{amsplain}
\providecommand{\bysame}{\leavevmode\hbox to3em{\hrulefill}\thinspace}
\providecommand{\MR}{\relax\ifhmode\unskip\space\fi MR }
\providecommand{\MRhref}[2]{\href{http://www.ams.org/mathscinet-getitem?mr=#1}{#2}}
\providecommand{\href}[2]{#2}

\end{document}